\documentclass[10pt]{amsart}
\usepackage{times,amsmath,amsbsy,amssymb,amscd,mathrsfs}
\usepackage{booktabs,graphicx,subfigure,epstopdf,wrapfig,chemarrow}

\usepackage{algorithm2e} 
\usepackage{multicol,multirow}
\usepackage{mathtools}
\usepackage[usenames,dvipsnames,svgnames,table]{xcolor}
\usepackage[all]{xy}
\usepackage{wrapfig}
\usepackage{tcolorbox}
\usepackage{enumitem}

\usepackage{tikz,tikz-cd}
\usepackage[utf8]{inputenc}
\usepackage{pgfplots} 
\usepackage{pgfgantt}
\usepackage{pdflscape}
\pgfplotsset{compat=newest} 
\pgfplotsset{plot coordinates/math parser=false}
\newlength\fwidth

\definecolor{myBlue}{rgb}{0.0,0.0,0.55}
\usepackage[pdftex,colorlinks=true,citecolor=myBlue,linkcolor=myBlue]{hyperref}

\usepackage[hyperpageref]{backref}

\usepackage{comment,multicol,xspace}

  \newcounter{mnote}
  \let\oldmarginpar\marginpar
    \renewcommand\marginpar[1]{\-\oldmarginpar[\raggedleft\footnotesize #1]%
    {\raggedright\footnotesize #1}}

\newtheorem{theorem}{Theorem}[section]
\newtheorem{lemma}[theorem]{Lemma}
\newtheorem{corollary}[theorem]{Corollary}
\newtheorem{proposition}[theorem]{Proposition}

\newtheorem{example}[theorem]{Example}

\newtheorem{remark}[theorem]{Remark}

\newcommand{\dd}{\,{\rm d}}

\newcommand{\supp}{\operatorname{supp}}

\DeclareMathOperator*{\spa}{span}

\newcommand{\tr}{\operatorname{tr}}
\newcommand{\alt}{\operatorname{Alt}}

\newcommand{\vertiii}[1]{{\left\vert\kern-0.25ex\left\vert\kern-0.25ex\left\vert #1 
    \right\vert\kern-0.25ex\right\vert\kern-0.25ex\right\vert}}

\newcommand{\Oplus}{\ensuremath{\vcenter{\hbox{\scalebox{1.5}{$\oplus$}}}}}

\newcommand{\TetLattice}[2]{%
\begin{tikzpicture}[
  scale=.65,
  x={(1cm,0cm)},
  y={(0.38cm,0.24cm)},
  z={(0cm,1cm)}
]
\coordinate (v0) at (0,0,0);
\coordinate (v1) at (4,0,0);
\coordinate (v2) at (0,4,0);
\coordinate (v3) at (0,0,4);

\draw (v0)--(v1)--(v3)--cycle;
\draw[dashed] (v2)--(v0) (v2)--(v1) (v2)--(v3);

\foreach \a in {0,...,#1}{
  \pgfmathtruncatemacro{\Mb}{#1-\a}
  \foreach \b in {0,...,\Mb}{
    \pgfmathtruncatemacro{\d}{#1-\a-\b}
    \pgfmathtruncatemacro{\c}{0}
    \pgfmathtruncatemacro{\supp}
    {(\a>0)+(\b>0)+(\d>0)}
    \pgfmathtruncatemacro{\ell}{\supp-1}

    \coordinate (p) at
    ({4*\b/#1},0,{4*\d/#1});

    \ifnum\ell<#2
      \node[
        circle, fill=gray!75, draw=none,
        minimum size=4.0pt, inner sep=0pt
      ] at (p) {};
    \else
      \ifnum\ell=1
        \node[
          circle, fill=black, draw=none,
          minimum size=3.8pt, inner sep=0pt
        ] at (p) {};
      \else\ifnum\ell=2
        \node[
          rectangle, fill=black, draw=none,
          minimum size=4.5pt, inner sep=0pt
        ] at (p) {};
      \else
        \node[
          circle, fill=black, draw=none,
          minimum size=3.2pt, inner sep=0pt
        ] at (p) {};
      \fi\fi
    \fi
  }
}
\foreach \a in {1,...,#1}{
  \pgfmathtruncatemacro{\Mb}{#1-\a}
  \ifnum\Mb>0
    \foreach \b in {1,...,\Mb}{
      \pgfmathtruncatemacro{\Mc}{#1-\a-\b}
      \ifnum\Mc>0
        \foreach \c in {1,...,\Mc}{
          \pgfmathtruncatemacro{\d}{#1-\a-\b-\c}
          \ifnum\d>0
            \coordinate (p) at
            ({4*\b/#1},{4*\c/#1},{4*\d/#1});
\node[
  regular polygon,
  regular polygon sides=3,
  fill=black,
  draw=none,
  minimum size=8pt,
  inner sep=0pt
] at (p) {};
          \fi
        }
      \fi
    }
  \fi
}
\end{tikzpicture}%
}

\newcommand{\TetInteriorLattice}[1]{%
\begin{tikzpicture}[
  scale=.65,
  x={(1cm,0cm)},
  y={(0.38cm,0.24cm)},
  z={(0cm,1cm)}
]
\coordinate (v0) at (0,0,0);
\coordinate (v1) at (4,0,0);
\coordinate (v2) at (0,4,0);
\coordinate (v3) at (0,0,4);

\draw (v0)--(v1)--(v3)--cycle;
\draw[dashed] (v2)--(v0) (v2)--(v1) (v2)--(v3);

\foreach \a in {1,...,#1}{
  \pgfmathtruncatemacro{\Mb}{#1-\a}
  \ifnum\Mb>0
    \foreach \b in {1,...,\Mb}{
      \pgfmathtruncatemacro{\Mc}{#1-\a-\b}
      \ifnum\Mc>0
        \foreach \c in {1,...,\Mc}{
          \pgfmathtruncatemacro{\d}{#1-\a-\b-\c}
          \ifnum\d>0
            \coordinate (p) at
            ({4*\b/#1},{4*\c/#1},{4*\d/#1});
\node[
  regular polygon,
  regular polygon sides=3,
  fill=black,
  draw=none,
  minimum size=8pt,
  inner sep=0pt
] at (p) {};
          \fi
        }
      \fi
    }
  \fi
}

\end{tikzpicture}%
}
\ifpdf
  \DeclareGraphicsExtensions{.eps,.pdf,.png,.jpg}
\else
  \DeclareGraphicsExtensions{.eps}
\fi

\newcommand{\Needspace}[1]{\par\vspace{#1}}

\begin{document}
\title[Whitney--Lagrange Decomposition]{Whitney--Lagrange Decomposition for the First Family of Finite Element Differential Forms}
\author{Long Chen}%
 \address{Department of Mathematics, University of California at Irvine, Irvine, CA 92697, USA}%
 \email{chenlong@math.uci.edu}%
 \author{Xuehai Huang}%
 \address{Corresponding author. School of Mathematics, Shanghai University of Finance and Economics, Shanghai 200433, China}%
 \email{huang.xuehai@sufe.edu.cn}%

 \thanks{The first author was supported by NSF DMS-2309777 and DMS-2309785. The second author was supported by the National Natural Science Foundation of China (Grant No. 12671432).}

\makeatletter
\@namedef{subjclassname@2020}{\textup{2020} Mathematics Subject Classification}
\makeatother

\begin{abstract}
We derive polynomial bubble duality in barycentric coordinates and construct explicit nodal bases for the first family of finite element differential forms. Logarithmic coordinates give an elementary proof of the Whitney--bubble decomposition. The Hessian of negative entropy defines a metric whose normalized Hodge star maps bubble generators to complementary-degree polynomial forms. This gives a direct geometric description of the known dualities and proves nondegeneracy of their integral pairings. Combining Whitney forms with shifted Lagrange polynomials then gives bases dual to tangential point evaluations in every dimension and order. The construction extends to scalar nodes that are unisolvent on every face. The interpolant is independent of the roots and commutes with face restriction.  
\end{abstract}

\maketitle


\section{Introduction}

Finite element exterior calculus describes conforming discretizations of the de Rham complex through the full spaces $\mathbb P_r\Lambda^k$ and the trimmed spaces $\mathbb P_r^-\Lambda^k$. These include the classical Lagrange, N\'ed\'elec, Raviart--Thomas, and Brezzi--Douglas--Marini elements; see~\cite{arnoldFiniteElementExterior2018,ArnoldFalkWinthe2006Finite}. We use geometric decompositions to derive the bubble dualities in barycentric coordinates and to construct explicit nodal bases for the trimmed family.

Arnold, Falk, and Winther gave rooted Whitney--bubble bases and consistent face extensions~\cite[Proposition~6.5 and Theorems~7.1--7.3]{ArnoldFalkWinther2009}. Related scalar--Whitney representations, bases, and extensions are developed in~\cite{ChristiansenRapetti2016,Licht2022,Isaac2022}. We first give an elementary proof of this decomposition.

Let $T$ be a $d$-simplex with barycentric coordinates $\lambda_i$, and let $r\geq0$ and $0\leq k\leq d$. For an $\ell$-face $f$ with $k\leq\ell\leq d$, let $\Delta_k(f)$ be its set of $k$-faces and choose a root vertex $f(0)$. Write $b_e=\prod_{i\in e}\lambda_i$ for the bubble polynomial of a face $e$, and let $w_e$ be its Whitney form. The face-associated bubble space is
\[
\mathbb B_{r+1}^-\Lambda^k(T; f)=\bigoplus_{\substack{e\in\Delta_k(f)\\ f(0)\in e}}\mathbb P_{r+k-\ell}(f)\,b_{f\setminus e}\,w_e.
\]
Here polynomials on $f$ are extended to $T$ using their Bernstein representations. Polynomial spaces of negative degree are zero. The geometric decomposition is
\begin{equation}\label{intro:dec}
\mathbb P_{r+1}^-\Lambda^k(T)=\bigoplus_{\ell=k}^{d}\bigoplus_{f\in\Delta_\ell(T)}\mathbb B_{r+1}^-\Lambda^k(T;f).
\end{equation}

Our first contribution is a direct barycentric derivation of bubble duality. The logarithmic coordinates are the gradient of the negative entropy $\sum_i\lambda_i\log\lambda_i$. Its Hessian defines the metric
\[
M=\sum_{i=0}^d\frac{\dd\lambda_i\otimes\dd\lambda_i}{\lambda_i}.
\]
Write $\mathbb B_r\Lambda^k(T)$ and $\mathbb B_{r+1}^-\Lambda^k(T)$ for the full and trimmed spaces with zero boundary trace. For $0\leq k<d$, set $k^*=d-k$ and $s=r-k^*$. With $b_T=\prod_{i=0}^d\lambda_i$, the normalized Hodge star $b_T^{-1/2}\star_M$ gives
\[
\begin{aligned}
\mathbb B_{r+1}^-\Lambda^k(T)&\cong\mathbb P_s\Lambda^{k^*}(T),
&&s\geq0,\\
\mathbb B_r\Lambda^k(T)&\cong\mathbb P_s^-\Lambda^{k^*}(T),
&&s\geq1.
\end{aligned}
\]
The same map makes $\int_T\omega\wedge(b_T^{-1/2}\star_M\omega)$ positive for every nonzero bubble $\omega$. Thus both wedge pairings are nondegenerate and give integral degrees of freedom (Theorem~\ref{thm:bubble-duality}). These dualities are known from Berchenko--Kogan's sphere construction~\cite[Corollaries~3.3--3.4]{BerchenkoKogan2021}. Our derivation identifies their action on each geometric face block directly from the logarithmic Whitney identity and the metric, and provides a more direct alternative construction.

Our second contribution is an explicit basis for the trimmed spaces dual to tangential point evaluations. Set $m=r+k+1$ and assume that the scalar nodes are unisolvent for $\mathbb P_m$ on every face, including $T$. Let $x_\alpha$ lie in the relative interior of a face $f$ with $\dim f\geq k$, and let $\phi_\alpha$ be its scalar Lagrange basis function on $T$. Orient each rooted $k$-face as $e=[f(0),e(1),\ldots,e(k)]$, with the remaining vertices in increasing order, and define
\[
\Phi_{\alpha,e}^{f,T}=\left(\prod_{i\in e\setminus\{f(0)\}}\lambda_i(x_\alpha)\right)\frac{\phi_\alpha}{b_e}\,w_e.
\]
Scalar unisolvence on every face ensures that $b_e$ divides $\phi_\alpha$, so the quotient is a polynomial of degree at most $r$. The division means polynomial cancellation, including at boundary points. The dual functional is
\[
N_{\alpha,e}^f(\omega)=(\tr_f\omega)(x_\alpha)\bigl((\texttt v_{e(1)}-x_\alpha)\wedge\cdots\wedge(\texttt v_{e(k)}-x_\alpha)\bigr).
\]
Using all scalar nodes in the relative interiors of faces of dimension at least $k$, and all $k$-faces containing the chosen root, gives a basis of $\mathbb P_{r+1}^-\Lambda^k(T)$ dual to these evaluations. Theorem~\ref{thm:nodal-whitney-lagrange} proves the formula on the principal lattice; Theorem~\ref{thm:scalar-node-transfer} extends it to scalar nodes that are unisolvent on every face.
The interpolant is independent of the root and compatible with face restriction (Proposition~\ref{prop:root-trace}). Shared face data assemble into the standard conforming trimmed space. For $k=0$, the construction reduces to scalar Lagrange interpolation.

The construction parallels the tangential--normal decomposition for the full family in~\cite{ChenHuangtnform}: both combine scalar Lagrange functions with face geometry, using Whitney forms here and tangential--normal frames for the full family.

Graglia, Wilton, and Peterson constructed interpolatory vector bases from lowest-order vector functions and shifted Lagrange polynomials~\cite{GragliaWiltonPeterson1997}. For tetrahedral one-forms, our formula has the same scalar--Whitney product, up to normalization and the choice of independent components; see Remark~\ref{rem:product-comparison} and~\cite[Section~3]{HooghiemstraHeulVuik2010}. Hiptmair uses moment degrees of freedom and discusses point-value alternatives~\cite{Hiptmair2001}. The small-simplex integrals in~\cite{Rapetti2007,RapettiBossavit2009,AlonsoRodriguezRapetti2019,AlonsoCamanoDeLosSantosRapetti2024} define a different interpolation system. The present construction treats every dimension and form degree with the same basis formula and explicit dual point evaluations.

Point evaluation is defined here for forms with continuous coefficients but is generally not bounded on $H\Lambda^k$. The interpolant does not in general commute with the exterior derivative, and cardinality alone gives no stability bound as the order grows. Related stability questions for integral interpolation are studied in~\cite{AlonsoBruniRapetti2022,BruniBrunoPiazzon2025}. Our basis simplifies the imposition of tangential boundary data. Nonuniform scalar nodes are allowed, but conditioning also depends on the geometry and normalization.

Section~\ref{sec:diffform} recalls trimmed spaces and Whitney forms. Section~\ref{sec:geomdecompTrimPolyForm} proves the geometric decomposition. Section~\ref{sec:rooted-bubbles} derives the barycentric metric, its Hodge star, and the bubble dualities. Section~\ref{sec:nodal-whitney-lagrange} constructs the dual nodal bases, proves their interpolation properties, and gives examples.

\section{Trimmed polynomial forms and Whitney forms}
\label{sec:diffform}

Let $T=[\texttt v_0,\ldots,\texttt v_d]\subset\mathbb R^d$ be a
nondegenerate simplex with barycentric coordinates
$\lambda_0,\ldots,\lambda_d$. Let $\Delta_\ell(T)$ be its set of
$\ell$-faces, and $\Delta(T)$ its set of nonempty faces. We identify a
face with its vertex set, ordered increasingly unless stated otherwise.
For $f=[f(0),\ldots,f(\ell)]\in\Delta_\ell(T)$, set
\[
f^*=\{0,\ldots,d\}\setminus f,\qquad
b_f=\prod_{i\in f}\lambda_i,\qquad
\dd\lambda_f=\dd\lambda_{f(0)}\wedge\cdots\wedge\dd\lambda_{f(\ell)}.
\]
Use $b_\varnothing=1$ and $\dd\lambda_\varnothing=1$; a consecutive
sublist is written $f(i:j)$. The opposite facet is $F_i=\{i\}^*$.
Traces are pullbacks under face inclusions, and
$\dd_f\lambda_i=\operatorname{tr}_f\dd\lambda_i$.
Let $\mathscr T^f$ be the tangent space of $f$ and
$\operatorname{Alt}^k(\mathscr T^f)$ its space of alternating $k$-forms.
Polynomial spaces of negative degree and form spaces outside the
range $0\leq k\leq\dim f$ are zero.

\subsection{Simplicial lattices}\label{sec:simplicial-lattices}

For an integer $m\geq0$, define the simplicial lattice
\[
\mathbb T_m^d:=\{\alpha\in\mathbb N_0^{d+1}:|\alpha|=m\},
\qquad
\operatorname{supp}\alpha:=\{i:\alpha_i>0\},
\qquad
\lambda^\alpha:=\prod_{i=0}^d\lambda_i^{\alpha_i}.
\]
For $f\in\Delta_\ell(T)$, let
\[
\mathbb T_m^\ell(f):=\{\alpha\in\mathbb T_m^d:\operatorname{supp}\alpha\subseteq f\},
\qquad
\mathbb T_m^\ell(\mathring f):=\{\alpha\in\mathbb T_m^d:\operatorname{supp}\alpha=f\}.
\]
For $m\geq1$, the geometric realization is
\[
\begin{gathered}
x_\alpha:=\frac1m\sum_{i=0}^d\alpha_i\texttt v_i,\\
x_\alpha\in f\iff\operatorname{supp}\alpha\subseteq f,
\qquad
x_\alpha\in\mathring f\iff\operatorname{supp}\alpha=f.
\end{gathered}
\]
Here $f$ denotes the geometric simplex in statements about $x_\alpha$ and its vertex set in statements about $\operatorname{supp}\alpha$. The same lattice notation applies to local multi-indices on a face.

For $m\geq0$, the monomials $\{\lambda^\alpha:\alpha\in\mathbb T_m^\ell(f)\}$
form a Bernstein basis of $\mathbb P_m(f)$. Thus every $p\in\mathbb P_m(f)$
has a unique representation
\[
p=\sum_{\alpha\in\mathbb T_m^\ell(f)}c_\alpha\lambda^\alpha.
\]
Using the same barycentric expression on $T$ defines its degree-$m$
Bernstein extension. This extension depends on the chosen degree $m$;
even a constant on $f$ can extend to a nonconstant polynomial on $T$.

\subsection{Trimmed spaces}

Let $X_T(x)=x-x_T$, where $x_T$ is the barycenter of $T$.
For $p\geq1$,
define the Koszul operator
\[
\kappa_p:\Lambda^p(T)\longrightarrow\Lambda^{p-1}(T),
\qquad
\kappa_p\omega:=\iota_{X_T}\omega,
\]
and set $\kappa_0=0$. We omit the degree subscript when it is clear. At $x\in T$,
\[
(\kappa_p\omega)_x(\boldsymbol v_1,\ldots,\boldsymbol v_{p-1})
=
\omega_x(X_T(x),\boldsymbol v_1,\ldots,\boldsymbol v_{p-1}).
\]
Homogeneity is relative to $x_T$, i.e., $\mathbb H_r(T)=\operatorname{span}\{(x-x_T)^\alpha:|\alpha|=r\}$, and 
$\mathbb H_r\Lambda^k(T)$ denotes forms with homogeneous degree-$r$
coefficients. The identities
\begin{align}\label{eq:kappa-barycentric}
\kappa\dd\lambda_i&=\lambda_i-\frac1{d+1},\\
\label{eq:kappawedge}
\kappa(\omega\wedge\eta)
&=(\kappa\omega)\wedge\eta+(-1)^p\omega\wedge\kappa\eta,
\qquad \omega\in\Lambda^p(T),
\end{align}
follow by direct calculation and will be used below. Cartan's identity gives
$\kappa\dd+\dd\kappa=(r+k)\operatorname{id}$ on
$\mathbb H_r\Lambda^k(T)$; see
\cite{arnoldFiniteElementExterior2018,ArnoldFalkWinthe2006Finite}.
For $r\geq0$ and $0\leq k\leq d$, the trimmed space is
\begin{equation}\label{eq:trimmed-def}
\mathbb P_{r+1}^-\Lambda^k(T)
=\mathbb P_r\Lambda^k(T)\oplus\kappa\mathbb H_r\Lambda^{k+1}(T).
\end{equation}
Changing the center changes the Koszul term only by a form in
$\mathbb P_r\Lambda^k(T)$, so the space is independent of that choice.
At the endpoint form degrees $k=0$ and $k=d$,
\[
\mathbb P_{r+1}^-\Lambda^0(T)=\mathbb P_{r+1}(T),\qquad
\mathbb P_{r+1}^-\Lambda^d(T)=\mathbb P_r\Lambda^d(T).
\]

\subsection{Whitney forms and traces}

For an oriented $k$-face $f$, define the Whitney form
\cite{Whitney1957}
\begin{equation}\label{eq:whitney}
w_f=\sum_{i=0}^k(-1)^i\lambda_{f(i)}
\dd\lambda_{f(0)}\wedge\cdots\wedge
\widehat{\dd\lambda_{f(i)}}\wedge\cdots\wedge\dd\lambda_{f(k)},
\end{equation}
where the hat denotes an omitted factor. Equations~\eqref{eq:kappa-barycentric} and \eqref{eq:kappawedge} give
\begin{equation}\label{wf}
w_f
=
\kappa\dd\lambda_f
+
\frac1{d+1}
\sum_{i=0}^k(-1)^i
\dd\lambda_{f(0)}\wedge\cdots
\wedge\widehat{\dd\lambda_{f(i)}}\wedge\cdots
\wedge\dd\lambda_{f(k)}.
\end{equation}
In particular, $w_f\in\mathbb P_1^-\Lambda^k(T)$.

\begin{lemma}\label{lm:dlambda-whitney}
Let $e=[e(0),\ldots,e(k)]$ be an oriented $k$-face, with
$0\leq k\leq d-1$. Orient $j\sqcup e$ as
$[j,e(0),\ldots,e(k)]$ for each $j\in e^*$. Then
\[
\dd\lambda_e=\sum_{j\in e^*}w_{j\sqcup e}.
\]
\end{lemma}
\begin{proof}
Expanding the Whitney forms gives
\[
\begin{aligned}
\sum_{j\in e^*}w_{j\sqcup e}
={}&\left(\sum_{j\in e^*}\lambda_j\right)\dd\lambda_e\\
&+\sum_{i=0}^{k}(-1)^{i+1}\lambda_{e(i)}
\left(\sum_{j\in e^*}\dd\lambda_j\right)
\wedge\dd\lambda_{e(0)}\wedge\cdots
\wedge\widehat{\dd\lambda_{e(i)}}\wedge\cdots
\wedge\dd\lambda_{e(k)}.
\end{aligned}
\]
Since $\sum_{j\in e^*}\dd\lambda_j=-\sum_{h=0}^{k}\dd\lambda_{e(h)}$,
only $h=i$ contributes to the $i$-th wedge product. Thus
\[
\sum_{j\in e^*}w_{j\sqcup e}
=\left(\sum_{j\in e^*}\lambda_j+\sum_{i=0}^{k}\lambda_{e(i)}\right)
\dd\lambda_e
=\dd\lambda_e.
\]
\end{proof}


For $e=[e(0),\ldots,e(k)]\in\Delta_k(f)$, the ambient and intrinsic Whitney forms are
\[
\begin{aligned}
w_e&=\sum_{i=0}^k(-1)^i\lambda_{e(i)}\,\dd\lambda_{e(0)}\wedge\cdots\wedge\widehat{\dd\lambda_{e(i)}}\wedge\cdots\wedge\dd\lambda_{e(k)},\\
w_e^f=\operatorname{tr}_f w_e&=\sum_{i=0}^k(-1)^i\lambda_{e(i)}\,\dd_f\lambda_{e(0)}\wedge\cdots\wedge\widehat{\dd_f\lambda_{e(i)}}\wedge\cdots\wedge\dd_f\lambda_{e(k)},
\end{aligned}
\]
where $\dd_f\lambda_i=\operatorname{tr}_f\dd\lambda_i$.
Under the Euclidean identification of one-forms with vectors, the proxies of $\dd\lambda_i$ and $\dd_f\lambda_i$ are $\nabla\lambda_i$ and $\nabla_f\lambda_i$, respectively. Here $\nabla_f\lambda_i$ is the orthogonal projection of $\nabla\lambda_i$ onto the tangent space of $f$, i.e., the surface gradient of $\lambda_i$.
Lemma \ref{lm:dlambda-whitney}, applied on $f\in\Delta_k(T)$, implies that $w_f^f\in \Lambda^k(f)$ is a volume form on $f$ and
\begin{equation}\label{eq:whitney-integral}
w_f^f=\dd_f\lambda_{f(1:k)},\qquad
\int_f w_f^f=\frac1{k!}.
\end{equation}
The normalization follows by affine pullback to the reference simplex.
For $k=0$, both statements mean evaluation at the vertex, where
$w_f^f=1$. In particular,
$w_T=\dd\lambda_1\wedge\cdots\wedge\dd\lambda_d$ is constant.

For $f=[f(0),\ldots,f(\ell)]$, write
\[
\Delta_k^0(f)=\{e\in\Delta_k(f):f(0)\in e\}.
\]
The root $f(0)$ selects $\binom{\ell}{k}$ forms. For $f=T$,
\begin{equation}\label{eq:Altbasis}
\{\dd\lambda_{e(1:k)}:e\in\Delta_k^0(T)\}
\quad\text{is a basis of }\operatorname{Alt}^k(\mathscr T^T),
\end{equation}
because $\dd\lambda_1,\ldots,\dd\lambda_d$ are independent. Likewise, with $k^*=d-k$,
\[
\{\dd\lambda_{e^*}:e\in\Delta_k^0(T)\}
\quad\text{is a basis of }\operatorname{Alt}^{k^*}(\mathscr T^T),
\]
since $e^*$ runs through all $k^*$-element subsets of $\{1,\ldots,d\}$.
The identity $\sum_{i=0}^d\lambda_i=1$ gives the linear dependence $\sum_{i=0}^d\dd\lambda_i=0$. Choosing vertex $0$ as the root specifies which covector to omit. Geometrically, taking $\texttt v_0$ as the affine origin gives the edge basis $\{\texttt v_i-\texttt v_0\}_{i=1}^d$, whose dual basis is $\{\dd\lambda_i\}_{i=1}^d$. Any vertex can serve as the root.

\subsection{Polynomial multiples}

\begin{lemma}\label{lm:qwf}
If $q\in\mathbb P_r(T)$ and $f\in\Delta_k(T)$, then
$qw_f\in\mathbb P_{r+1}^-\Lambda^k(T)$.
\end{lemma}
\begin{proof}
Write $q=q_r+\widetilde q$, with $q_r\in\mathbb H_r(T)$ and
$\widetilde q\in\mathbb P_{r-1}(T)$. By~\eqref{wf}, the form $w_f-\kappa\dd\lambda_f$ is constant, so
\[
qw_f=\kappa(q_r\dd\lambda_f)
+\widetilde q\,\kappa\dd\lambda_f+q(w_f-\kappa\dd\lambda_f).
\]
The first term lies in $\kappa\mathbb H_r\Lambda^{k+1}(T)$, and the last two lie in $\mathbb P_r\Lambda^k(T)$.
For $r=0$, take $\widetilde q=0$.
\end{proof}

\begin{lemma}\label{lm:Pwspan}
For $r\geq0$ and $0\leq k\leq d$,
\begin{equation}\label{eq:Pwspan}
\mathbb P_{r+1}^-\Lambda^k(T)
=\sum_{f\in\Delta_k(T)}\mathbb P_r(T)w_f.
\end{equation}
\end{lemma}
\begin{proof}
Lemma~\ref{lm:qwf} gives one inclusion. The Whitney forms span
constant $k$-forms by Lemma~\ref{lm:dlambda-whitney}, or by
$\sum_i\lambda_i=1$ when $k=0$. Thus the right-hand side contains
$\mathbb P_r\Lambda^k(T)$. The forms $\dd\lambda_f$, $f\in\Delta_k(T)$, span the constant $(k+1)$-forms. By~\eqref{wf}, each $p\kappa\dd\lambda_f$, with $p\in\mathbb H_r(T)$, differs from $pw_f$ by a form in $\mathbb P_r\Lambda^k(T)$. Thus the right-hand side also contains $\kappa\mathbb H_r\Lambda^{k+1}(T)$. For $k=d$, this component is zero.
\end{proof}

Taking traces in~\eqref{eq:Pwspan} gives
\begin{equation}\label{eq:trimmed-face-trace}
\tr_F\mathbb P_{r+1}^-\Lambda^k(T)=\mathbb P_{r+1}^-\Lambda^k(F),
\qquad \dim F\geq k.
\end{equation}
Indeed, $\tr_F w_e=0$ if $e\nsubseteq F$ and $\tr_F w_e=w_e^F$ otherwise, and every scalar polynomial on $F$ has an extension of the same degree bound to $T$.

For $r=0$, the sum in~\eqref{eq:Pwspan} is direct. For $r>0$, it is generally redundant;
on an oriented triangle, for example,
$\lambda_iw_{j\ell}-\lambda_jw_{i\ell}+\lambda_\ell w_{ij}=0$.
Indeed, substituting $w_{ab}=\lambda_a\dd\lambda_b-\lambda_b\dd\lambda_a$ gives
\[
\begin{aligned}
&\lambda_i(\lambda_j\dd\lambda_\ell-\lambda_\ell\dd\lambda_j)
-\lambda_j(\lambda_i\dd\lambda_\ell-\lambda_\ell\dd\lambda_i)
+\lambda_\ell(\lambda_i\dd\lambda_j-\lambda_j\dd\lambda_i)=0,
\end{aligned}
\]
since the terms cancel in pairs. We will give a direct sum using bubble polynomial forms.

\section{Geometric decomposition of trimmed polynomial forms}\label{sec:geomdecompTrimPolyForm}

We recall the rooted Whitney--bubble decomposition of
\cite{ArnoldFalkWinther2009}. Logarithmic coordinates prove the independence of the rooted bubble forms. Face traces then give the geometric decomposition. We orient each rooted face as $f=[0,f(1),\ldots,f(k)]$.

\subsection{Bubble polynomial forms}\label{sec:bubblePolyForm}

We first construct trace-free barycentric forms. 

\begin{lemma}[Trace-free barycentric forms]\label{lem:bubble-dlambda}
Let $f\in\Delta_k(T)$ for $0\leq k\leq d$. Then the $(k+1)$-form $b_{f^*}\dd\lambda_f$ satisfies
\begin{equation}\label{eq:tracefree}
{\tr}_{\partial T}\bigl(b_{f^*}\dd\lambda_f\bigr)=0.
\end{equation}
In particular, the $k$-form
\[
b_{f(1:k)^*}\dd\lambda_{f(1:k)}
=
b_{f^*}\lambda_{f(0)}\dd\lambda_{f(1:k)}
\]
also has zero tangential trace on $\partial T$, with the empty wedge understood to be $1$.
\end{lemma}

\begin{proof}
Let $F_i$ be the facet opposite vertex $i$. Since
\[
\lambda_i|_{F_i}=0,
\qquad
{\tr}_{F_i}\dd\lambda_i
=
\dd_{F_i}(\lambda_i|_{F_i})
=
0,
\]
a form has zero tangential trace on $F_i$ if it contains $\lambda_i$ as a scalar factor or $\dd\lambda_i$ as a wedge factor. 
For $b_{f^*}\dd\lambda_f$, these cases correspond to $i\in f^*$ and $i\in f$, respectively. 

For $k=0$, the second assertion reduces to
${\tr}_{\partial T}b_T=0$, which is immediate. For $k\geq1$, apply
\eqref{eq:tracefree} to the face $f(1:k)$.
\end{proof}

Let $\Delta_k^0(T)$ denote the $k$-faces containing vertex $0$. The family 
\[
\{b_{f^*}\lambda_0\dd\lambda_{f(1:k)}:f\in\Delta_k^0(T)\}\subset \ker(\tr_{\partial T})
\] 
is linearly independent because the forms $\dd\lambda_{f(1:k)}$ form a basis of $\alt^k$and each scalar factor is a nonzero polynomial. These forms belong to $\mathbb P_{d-k+1}\Lambda^k(T)$. Replacing $\lambda_0\dd\lambda_{f(1:k)}$ by the Whitney form $w_f$ preserves the zero-trace property and gives forms in the trimmed family. We call the replacement
\[
\lambda_0\dd\lambda_{f(1:k)}\longmapsto w_f,
\qquad f\in\Delta_k^0(T),
\]
the \emph{Whitney completion} with root $0$. It adds the remaining terms in \eqref{eq:whitney} to the term with undifferentiated coordinate $\lambda_0$.

\begin{lemma}[Trace-free weighted Whitney forms]\label{lm:trbubble}
Let $0\leq k\leq d-1$. For every $f\in\Delta_k(T)$,
\[
b_{f^*}w_f
\in
\mathbb P_{d-k+1}^-\Lambda^k(T)
\cap
\ker({\tr}_{\partial T}),
\]
where $w_f$ is defined in \eqref{eq:whitney}.
\end{lemma}

\begin{proof}
\emergencystretch=1em
Since $\deg b_{f^*}=d-k$, membership in the trimmed space follows from Lemma~\ref{lm:qwf}. Let $F_i$ be the facet opposite vertex $i$. If $i\in f^*$, then $b_{f^*}$ contains $\lambda_i$. If $i\in f$, every term of $w_f$ contains either $\lambda_i$ or $\dd\lambda_i$. In both cases, the tangential trace on $F_i$ vanishes.
\end{proof}

On the interior of $T$, introduce the logarithmic barycentric coordinates
\[
z_i:=\log\frac{\lambda_i}{\lambda_0},
\qquad
i=1,\ldots,d.
\]
Their differentials are
\[
\dd z_i
=
\frac{\dd\lambda_i}{\lambda_i}
-
\frac{\dd\lambda_0}{\lambda_0}.
\]
For $1\leq i,j\leq d$, the affine barycentric identities give
\[
\dd z_j\bigl(\lambda_i(x)(\texttt v_i-x)\bigr)=\delta_{ij}.
\]
Hence $\dd z_1,\ldots,\dd z_d$ form a basis of the cotangent space at every interior point.

\begin{lemma}[Logarithmic representation of Whitney forms]\label{lm:whitney-log}
For $f\in\Delta_k^0(T)$, the following identities hold on $\mathring T$:
\begin{equation}\label{eq:whitney-log}
w_f=b_f\,\dd z_{f(1:k)},
\qquad
b_{f^*}w_f=b_T\,\dd z_{f(1:k)},
\end{equation}
where
\[
\dd z_{f(1:k)}
:=
\dd z_{f(1)}\wedge\cdots\wedge\dd z_{f(k)},
\qquad
\dd z_\varnothing:=1.
\]
\end{lemma}

\begin{proof}
Since $f(0)=0$,
\[
b_f\,\dd z_{f(1:k)}
=
\lambda_0\prod_{j=1}^k\lambda_{f(j)}
\bigwedge_{j=1}^k
\left(
\frac{\dd\lambda_{f(j)}}{\lambda_{f(j)}}
-
\frac{\dd\lambda_0}{\lambda_0}
\right).
\]
Terms containing two factors $\dd\lambda_0$ vanish. The term containing no $\dd\lambda_0$ is $\lambda_0\dd\lambda_{f(1:k)}$. Choosing $-\dd\lambda_0/\lambda_0$ from the $i$-th factor and moving it to the front gives
\[
(-1)^i\lambda_{f(i)}
\dd\lambda_0\wedge
\dd\lambda_{f(1)}\wedge\cdots\wedge
\widehat{\dd\lambda_{f(i)}}\wedge\cdots\wedge
\dd\lambda_{f(k)}.
\]
Their sum is $w_f$. Multiplying by $b_{f^*}$ gives the second identity because $b_fb_{f^*}=b_T$.
\end{proof}

The coordinates $z_i$ are defined only on $\mathring T$ and are singular at the boundary. By \eqref{eq:whitney-log},
\[
b_T\,\dd z_{f(1:k)}=b_{f^*}w_f
\qquad\text{on }\mathring T.
\]
The right-hand side is a polynomial form on $T$ with zero boundary trace. Thus, the weighted expression extends to all of $T$, and its apparent singularities are removable.

\begin{lemma}[Independence of weighted Whitney forms]\label{lm:wfbasis}
Let $0\leq k\leq d-1$. The family
\[
\left\{b_{f^*}w_f:f\in\Delta_k^0(T)\right\}
\]
is linearly independent over polynomial coefficients: if
\[
\sum_{f\in\Delta_k^0(T)}p_f b_{f^*}w_f=0
\]
for polynomials $p_f$, then every $p_f$ vanishes. Consequently,
\[
\spa\left\{b_{f^*}w_f:f\in\Delta_k^0(T)\right\}
\subseteq
\mathbb P_{d-k+1}^-\Lambda^k(T)
\cap
\ker({\tr}_{\partial T}),
\]
and this span has dimension $\binom{d}{k}$.
\end{lemma}

\begin{proof}
By \eqref{eq:whitney-log}, on $\mathring T$,
\[
\sum_{f\in\Delta_k^0(T)}p_f b_{f^*}w_f
=
b_T\sum_{f\in\Delta_k^0(T)}p_f\,\dd z_{f(1:k)}.
\]
Since $b_T>0$ and the forms $\dd z_{f(1:k)}$ form a basis at each interior point, every $p_f$ vanishes on $\mathring T$ and hence vanishes identically. The inclusion follows from Lemma~\ref{lm:trbubble}, and the dimension follows from $|\Delta_k^0(T)|=\binom{d}{k}$.
\end{proof}


For $0\leq k\leq d-1$ and $r\geq0$, define
\begin{equation}\label{eq:BT-}
\mathbb B_{r+1}^-\Lambda^k(T)
:=
\Oplus_{f\in\Delta_k^0(T)}
\mathbb P_{r+k-d}(T)b_{f^*}w_f,
\end{equation}
where polynomial spaces of negative degree are understood to be zero. The sum is direct by Lemma~\ref{lm:wfbasis}. If
$p\in\mathbb P_{r+k-d}(T)$, then $pb_{f^*}\in\mathbb P_r(T)$.
Lemmas~\ref{lm:qwf} and \ref{lm:trbubble} therefore give
\begin{equation}\label{eq:trimbubbleformk0}
\mathbb B_{r+1}^-\Lambda^k(T)
\subseteq
\mathbb P_{r+1}^-\Lambda^k(T)
\cap
\ker({\tr}_{\partial T}).
\end{equation}
This is the bubble basis of
\cite[Proposition~6.5]{ArnoldFalkWinther2009}.
The logarithmic proof above explains the polynomial
independence established in~\cite[Lemma~6.3]{ArnoldFalkWinther2009}.
Consequently, with the convention $\binom ab=0$ for integers $a\geq0$ whenever $b<0$ or $b>a$,
\begin{equation}\label{eq:dimB-}
\begin{aligned}
\dim\mathbb B_{r+1}^-\Lambda^k(T)
&=
\binom{d}{k}\dim\mathbb P_{r+k-d}(T)\\
&=
\dim\mathbb P_{r+k-d}\Lambda^k(T)
=
\binom{d}{k}\binom{r+k}{d}.
\end{aligned}
\end{equation}
We set $\mathbb B_{r+1}^-\Lambda^d(T)=\mathbb P_r\Lambda^d(T)$, since its boundary trace vanishes automatically.


\subsection{Full-family bubble factors}

Recall that $f^*$ is the complementary vertex set and $b_f=\prod_{i\in f}\lambda_i$. Choose a root $f(0)$ on each face; $\Delta_k^0(f)$ denotes its $k$-faces containing that root. Let $\mathscr N^f$ be the normal space of $f$ and $\star_T$ the Euclidean Hodge star. Polynomial spaces of negative degree are zero.

For $r\geq0$ and $0\leq k\leq d$, the full-family bubble space is
\[
\mathbb B_r\Lambda^k(T):=\mathbb P_r\Lambda^k(T)\cap\ker({\tr}_{\partial T}).
\]
For $0\leq k<d$, write $k^*=d-k$. The full-family bubble decomposition
\cite{ChenHuangtnform} is
\[
\mathbb B_r\Lambda^k(T)
=\Oplus_{\ell=k^*}^{d}\Oplus_{f\in\Delta_\ell(T)}
b_f\mathbb P_{r-\ell-1}(f)\otimes\star_T\alt^{k^*}(\mathscr T^f).
\]
For an $\ell$-face $f$, the tangential--normal decomposition gives
\[
\star_T\alt^{k^*}(\mathscr T^f)
=\alt^{d-\ell}(\mathscr N^f)\wedge\alt^{\ell-k^*}(\mathscr T^f).
\]
Choose a root $f(0)$. The normal factors are $\dd\lambda_i$, $i\in f^*$,
and the restrictions of $\dd\lambda_{f(i)}$, $i=1,\ldots,\ell$, form a
tangential basis. Wedging with all normal factors removes the normal
parts of the remaining covectors. Hence a basis of $\star_T\alt^{k^*}(\mathscr T^f)$ is
$\dd\lambda_{f^*\cup g}$, where $g\subseteq f(1:\ell)$ has
$\ell-k^*$ vertices. Set $e=f\setminus g\in\Delta_{k^*}^0(f)$. Then
\[
e^*=f^*\cup g,\qquad
\dd\lambda_{f^*\cup g}=\dd\lambda_{e^*},\qquad
b_f=b_{f\setminus e}b_e,
\]
where $e^*$ is ordered increasingly. Thus
\[
\star_T\alt^{k^*}(\mathscr T^f)
=\operatorname{span}\{\dd\lambda_{e^*}:e\in\Delta_{k^*}^0(f)\},
\qquad \dim\star_T\alt^{k^*}(\mathscr T^f)=\binom{\ell}{k^*},
\]
and the full bubble decomposition becomes
\begin{equation}\label{intro:BrT}
\mathbb B_r\Lambda^k(T)
=\Oplus_{\ell=k^*}^{d}\Oplus_{f\in\Delta_\ell(T)}
\Oplus_{e\in\Delta_{k^*}^0(f)}
b_{f\setminus e}\mathbb P_{r-\ell-1}(f)\,b_e\dd\lambda_{e^*}.
\end{equation}
The factor $b_e\dd\lambda_{e^*}$ already has zero boundary trace on $T$. The additional factor $b_{f\setminus e}$
selects the Bernstein support of the block indexed by $f$.

\subsection{Bubble polynomial spaces on subsimplices}
The preceding construction applies to any simplex. To obtain bubble spaces on an $\ell$-dimensional subsimplex $f$, we replace the ambient simplex $T$ by $f$, its dimension $d$ by $\ell$, and the indexing face by $e\subseteq f$.
\subsubsection{Intrinsic bubble spaces}
Let $f\in\Delta_\ell(T)$ with $0\leq k\leq\ell\leq d$. For $e\in\Delta_k(f)$, the relative bubble polynomial $b_{f\setminus e}$
has degree $\ell-k$. With $f(0)$ as the root, the intrinsic bubble space is
\begin{equation}\label{eq:intrinsicB-}
\mathbb B_{r+1}^-\Lambda^k(f)
:=
\bigoplus_{e\in\Delta_k^0(f)}
\mathbb P_{r+k-\ell}(f)b_{f\setminus e}w_e^f.
\end{equation}
Recall that $\Delta_k^0(f)$ consists of the $k$-faces of $f$ containing $f(0)$. Polynomial spaces of negative degree are understood to be zero.

For $\ell>k$, this space consists of trace-free trimmed forms on $f$. As a
consequence of Corollary~\ref{cor:trimbubbleequality} below,
\[
\mathbb B_{r+1}^-\Lambda^k(f)
=
\mathbb P_{r+1}^-\Lambda^k(f)
\cap
\ker({\tr}_{\partial f}).
\]
When $\ell=k$, the only choice is $e=f$, and $b_{f\setminus e}=b_\varnothing=1$. Hence,
\[
\mathbb B_{r+1}^-\Lambda^k(f)
=
\mathbb P_r(f)w_f^f
=
\mathbb P_{r+1}^-\Lambda^k(f)
=
\mathbb P_{r}\Lambda^k(f)
, \quad f\in \Delta_k(T).
\]
The boundary trace vanishes automatically because $\partial f$ has dimension less than $k$.

\subsubsection{Extension}\label{sec:bernstein-extension}
For $p\in\mathbb P_{r+k-\ell}(f)$, write its unique expansion in the Bernstein basis 
\[
p=\sum_{|\alpha_f|=r+k-\ell}c_{\alpha_f}\lambda_f^{\alpha_f},
\qquad
\lambda_f^{\alpha_f}:=\prod_{i\in f}\lambda_i^{\alpha_i}.
\]
Viewing each $\lambda_i$ as a barycentric coordinate on $T$ defines the extension by the same expression. Its restriction to $f$ is $p$, and its degree is at most $r+k-\ell$. This is a fixed-degree extension: even a constant on $f$ may extend to a nonconstant polynomial on $T$. If $r+k-\ell<0$, the coefficient space is zero and no extension is needed.
Define the face-associated space
\begin{equation}\label{eq:B-}
\mathbb B_{r+1}^-\Lambda^k(T;f)
:=
\Oplus_{e\in\Delta_k^0(f)}
\mathbb P_{r+k-\ell}(f)b_{f\setminus e}w_e.
\end{equation}
For each coefficient $p\in\mathbb P_{r+k-\ell}(f)$,
$p\,b_{f\setminus e}\in\mathbb P_r(T).$
Lemma \ref{lm:qwf} implies
$
\mathbb B_{r+1}^-\Lambda^k(T;f)
\subseteq
\mathbb P_{r+1}^-\Lambda^k(T).$

The tangential trace satisfies
${\tr}_f\bigl(p\,b_{f\setminus e}w_e\bigr)
=
p\,b_{f\setminus e}w_e^f.
$
Thus, it maps the ambient generators to the corresponding intrinsic generators. The directness of \eqref{eq:intrinsicB-} implies both the directness of \eqref{eq:B-} and the injectivity of this trace map. Consequently,
\begin{equation}\label{eq:tracebijectionB}
{\tr}_f:
\mathbb B_{r+1}^-\Lambda^k(T;f)
\longrightarrow
\mathbb B_{r+1}^-\Lambda^k(f)
\end{equation}
is an isomorphism. Its inverse extends each scalar coefficient by its Bernstein representation and replaces $w_e^f$ by $w_e$.

\subsubsection{Trace properties}
Unlike $b_{e^*}w_e$, the form $b_{f\setminus e}w_e$ contains only the bubble factor relative to $f$. It need not vanish on every facet of $T$, but its trace vanishes on every subsimplex that does not contain $f$.

\begin{lemma}\label{lem:facebubbletrace}
Let $f\in\Delta_\ell(T)$ and $e\in\Delta_k(f)$ with $0\leq k\leq\ell\leq d$. If $g\in\Delta(T)$ satisfies $\dim g\geq k$ and $f\nsubseteq g$, then
\[
{\tr}_g\bigl(b_{f\setminus e}w_e\bigr)=0.
\]
\end{lemma}

\begin{proof}
Choose a vertex $j\in f\setminus g$. If $j\in f\setminus e$, then $b_{f\setminus e}$ contains $\lambda_j$, so its restriction to $g$ vanishes. If $j\in e$, every term in $w_e$ contains either $\lambda_j$ or $\dd\lambda_j$. Since
\[
\lambda_j|_g=0,
\qquad
{\tr}_g\dd\lambda_j
=
\dd_g(\lambda_j|_g)
=
0,
\]
we have ${\tr}_g w_e=0$. In either case, the product has zero trace on $g$.
\end{proof}

\begin{corollary}\label{cor:trbubble}
Let $f\in\Delta_\ell(T)$ and $e\in\Delta_k(f)$ with $k<\ell$. Then:
\begin{enumerate}
\item ${\tr}_{\partial f}(b_{f\setminus e}w_e)=0$;
\item ${\tr}_g(b_{f\setminus e}w_e)=0$ for every $g\in\Delta_\ell(T)$ with $g\neq f$;
\item ${\tr}_g(b_{f\setminus e}w_e)=0$ for every $g\in\Delta(T)$ with $k\leq\dim g<\ell$.
\end{enumerate}
\end{corollary}

\begin{proof}
Each subsimplex $g$ under consideration does not contain $f$. Apply Lemma~\ref{lem:facebubbletrace}.
\end{proof}

If $F\in\Delta(T)$ contains $f$, however, then
${\tr}_F\bigl(b_{f\setminus e}w_e\bigr)\neq0.$
Indeed, restricting further to $f$ gives the nonzero form
$b_{f\setminus e}w_e^f.$
These trace properties explain why $b_{f\setminus e}w_e$ is associated with $f$. 

There are $\binom{\ell}{k}$ faces $e\in\Delta_k(f)$ containing $f(0)$. Therefore,
\begin{equation}\label{eq:dimBf-}
\dim\mathbb B_{r+1}^-\Lambda^k(T;f)
=
\dim\mathbb B_{r+1}^-\Lambda^k(f)
=
\binom{\ell}{k}\binom{r+k}{\ell}.
\end{equation}
In particular, $\mathbb B_{r+1}^-\Lambda^k(T;f)=\{0\}$ when $\ell>r+k$. Thus, nonzero bubble $k$-forms on an $\ell$-dimensional face require polynomial order $r+1\geq\ell-k+1$. For interior bubbles, where $\ell=d$, the minimum order is $d-k+1$.

\subsection{Geometric decomposition}
The face blocks now combine to give the full trimmed space. Their trace structure proves directness, and a dimension count proves that they span.
\begin{theorem}[Geometric decomposition;~\cite{ArnoldFalkWinther2009}]\label{th:trimdec}
For $r\geq0$ and $0\leq k\leq d$,
\begin{equation}\label{eq:trimgeodecomp}
\begin{aligned}
\mathbb P_{r+1}^-\Lambda^k(T)&=\Oplus_{\ell=k}^{d}\Oplus_{f\in\Delta_\ell(T)}\mathbb B_{r+1}^-\Lambda^k(T;f)\\
& = \Oplus_{\ell=k}^{d}\Oplus_{f\in\Delta_\ell(T)}\Oplus_{e\in\Delta_k^0(f)}
\mathbb P_{r+k-\ell}(f)b_{f\setminus e}w_e.
\end{aligned}
\end{equation}
In particular, when $r=0$,
\[
\mathbb P_1^-\Lambda^k(T)
=
\bigoplus_{f\in\Delta_k(T)}\mathbb R\,w_f,
\]
which is the lowest-order space of Whitney forms.
\end{theorem}

\begin{proof}
For $k=d$, the right-hand side of~\eqref{eq:trimgeodecomp} consists only of
the element-interior block
$\mathbb B_{r+1}^-\Lambda^d(T;T)=\mathbb P_r(T)w_T
=\mathbb P_{r+1}^-\Lambda^d(T)$, so the result is immediate. We therefore
assume $0\leq k\leq d-1$.

Use the face traces as form-valued degrees of freedom and order the trace rows and face blocks by increasing face dimension, with the same face order within each dimension. Lemma~\ref{lem:facebubbletrace} gives
\begin{equation}\label{eq:lowertriangular}
\left.\tr_g\right|_{\mathbb B_{r+1}^-\Lambda^k(T;f)}=0
\qquad\text{if }f\nsubseteq g.
\end{equation}
Thus the trace maps form a block lower-triangular array. The diagonal maps
\[
\left.\tr_f\right|_{\mathbb B_{r+1}^-\Lambda^k(T;f)}:
\mathbb B_{r+1}^-\Lambda^k(T;f)
\longrightarrow\mathbb B_{r+1}^-\Lambda^k(f)
\]
are isomorphisms by~\eqref{eq:tracebijectionB}.

Suppose $\sum_f\omega_f=0$, where $\omega_f\in\mathbb B_{r+1}^-\Lambda^k(T;f)$. Taking traces gives
\[
0=\tr_g\Bigl(\sum_f\omega_f\Bigr)
=\sum_{f\subseteq g}\tr_g\omega_f.
\]
Starting with the $k$-faces, proceed in increasing face dimension. Once all lower-dimensional components vanish, the equation reduces to $\tr_g\omega_g=0$. Injectivity of the diagonal trace map gives $\omega_g=0$. Hence the sum is direct.

It remains to compare dimensions. By~\eqref{eq:dimBf-},
\begin{align*}
\sum_{\ell=k}^d\binom{d+1}{\ell+1}\binom{r+k}{\ell}\binom{\ell}{k}
&=\binom{r+k}{k}\sum_{\ell=k}^d\binom{d+1}{\ell+1}\binom{r}{r+k-\ell}\\
&=\binom{r+d+1}{r+k+1}\binom{r+k}{k} =\dim\mathbb P_{r+1}^-\Lambda^k(T).
\end{align*}
The second identity is Vandermonde's identity, obtained by comparing the coefficient of $x^{r+k+1}$ in
\[
(1+x)^{d+1}(1+x)^r=(1+x)^{r+d+1}.
\]
The last equality uses the standard dimension formula for trimmed forms~\cite[(3.7)]{ArnoldFalkWinther2009}, with order $r+1$.
\end{proof}

The same triangular structure identifies the interior block: zero boundary trace eliminates every proper-face component.

\begin{corollary}\label{cor:trimbubbleequality}
For $r\geq0$ and $0\leq k\leq d$,
\begin{equation}\label{eq:trimbubbleequality}
\mathbb B_{r+1}^-\Lambda^k(T)=\mathbb P_{r+1}^-\Lambda^k(T)\cap\ker(\tr_{\partial T}).
\end{equation}
\end{corollary}

\begin{proof}
For $k=d$, both sides of~\eqref{eq:trimbubbleequality} equal
$\mathbb P_r\Lambda^d(T)$, because the boundary trace of a top-degree form
vanishes automatically. We therefore assume $0\leq k\leq d-1$.

The inclusion from left to right is~\eqref{eq:trimbubbleformk0}. Conversely, let $\omega\in\mathbb P_{r+1}^-\Lambda^k(T)$ have zero boundary trace and write $\omega=\sum_f\omega_f$ by~\eqref{eq:trimgeodecomp}. For every proper face $g$ with $\dim g\geq k$, choose a facet $F$ containing $g$. Transitivity of traces gives
\[
\tr_g\omega=\tr_g^F(\tr_F\omega)=0,
\]
where $\tr_g^F$ is the trace from $F$ to $g$. Hence
\[
0=\tr_g\Bigl(\sum_f\omega_f\Bigr)
=\sum_{f\subseteq g}\tr_{g}\omega_f.
\]
The same forward substitution as in the proof of Theorem~\ref{th:trimdec} gives $\omega_f=0$ for every proper face $f$. Thus $\omega=\omega_T\in\mathbb B_{r+1}^-\Lambda^k(T)$ by~\eqref{eq:B-} and~\eqref{eq:BT-}.
\end{proof}

\section{Geometric duality of polynomial bubble spaces}\label{sec:rooted-bubbles}

The trimmed bubble basis is indexed by rooted Whitney forms. We now express the full bubble basis using complementary barycentric forms. A weighted Hodge star maps these generators to polynomial forms of complementary degree and gives nondegenerate integral pairings. 

\subsection{Logarithmic coordinates and the barycentric metric}
\label{sec:barycentric-metric}

We work on $\mathring T$, where all barycentric coordinates are positive,
and use the orientation
$\Omega_\lambda=\dd\lambda_1\wedge\cdots\wedge\dd\lambda_d$.
Recall $z_i=\log(\lambda_i/\lambda_0)$ from
Section~\ref{sec:bubblePolyForm}. These logarithmic coordinates are the
partial derivatives of the negative entropy
\[
E(\lambda_1,\ldots,\lambda_d)=\sum_{i=0}^d\lambda_i\log\lambda_i,
\qquad \lambda_0=1-\sum_{i=1}^d\lambda_i.
\]
Writing $\boldsymbol\lambda=(\lambda_1,\ldots,\lambda_d)^{\mathsf T}$
and $\boldsymbol z=(z_1,\ldots,z_d)^{\mathsf T}$, we have
\[
\boldsymbol z=\nabla_{\boldsymbol\lambda}E,
\qquad
\dd\boldsymbol z=\mathsf M\dd\boldsymbol\lambda, 
\qquad
\mathsf M=D^{-1}+\lambda_0^{-1}\mathbf1\mathbf1^{\mathsf T},
\]
where $D=\operatorname{diag}(\lambda_1,\ldots,\lambda_d)$ and
$\mathbf1$ is the column of ones. Indeed, for $i=1, \ldots, d$,
\[
\frac{\partial E}{\partial\lambda_i}= \log(\lambda_i/\lambda_0) = z_i,
\qquad
\mathsf M_{ij} = \frac{\partial^2E}{\partial\lambda_j\partial\lambda_i}
=\frac{\delta_{ij}}{\lambda_i}+\frac1{\lambda_0}.
\]
The Hessian defines the barycentric metric
\begin{equation}\label{eq:barycentric-metric}
M=\sum_{i=0}^d\frac{\dd\lambda_i\otimes\dd\lambda_i}{\lambda_i}.
\end{equation}
To express it as a matrix, use the affine tangent basis
$\boldsymbol t_i=\texttt v_i-\texttt v_0$, $i=1,\ldots,d$. Since
$\dd\lambda_j(\boldsymbol  t_i)=\delta_{ji}$ for $j\geq1$ and
$\dd\lambda_0(\boldsymbol t_i)=-1$, its matrix is
\[
\mathsf M_{ij}:=M( \boldsymbol t_i, \boldsymbol t_j)
=\frac{\delta_{ij}}{\lambda_i}+\frac1{\lambda_0},
\]
The inverse metric
acts on covectors: if $\alpha=\sum_i a_i\dd\lambda_i$ and
$\beta=\sum_i b_i\dd\lambda_i$ have coefficient columns $a,b$, then
\[
\langle\alpha,\beta\rangle_{M^{-1}}
=a^{\mathsf T}\mathsf M^{-1}b.
\]

\begin{lemma}[Metric duality]\label{lm:barycentric-metric}
The metric $M$ is positive definite on $\mathring T$, and
\[
\mathsf M^{-1}=D-\boldsymbol\lambda\boldsymbol\lambda^{\mathsf T},
\qquad \det\mathsf M=b_T^{-1}.
\]
The logarithmic and barycentric covectors are $M$-dual:
\begin{equation}\label{eq:metric-covector-duality}
\langle\dd\lambda_j,\dd z_i\rangle_{M^{-1}}=\delta_{ji}.
\end{equation}
\end{lemma}
\begin{proof}
For a nonzero tangent vector $u$, some $\dd\lambda_i(u)$ is nonzero,
so $M(u,u)>0$. Using $D^{-1}\boldsymbol\lambda=\mathbf1$ and
$\mathbf1^{\mathsf T}\boldsymbol\lambda=1-\lambda_0$, multiplication gives
\[
\mathsf M(D-\boldsymbol\lambda\boldsymbol\lambda^{\mathsf T})
=I+\bigl[-1+\lambda_0^{-1}-\lambda_0^{-1}(1-\lambda_0)\bigr]
\mathbf1\boldsymbol\lambda^{\mathsf T}=I.
\]
The determinant formula for a rank-one update gives
\[
\det\mathsf M
=\left(\prod_{i=1}^d\lambda_i^{-1}\right)
\left(1+\frac{\sum_{i=1}^d\lambda_i}{\lambda_0}\right)=b_T^{-1}.
\]
Finally, $\dd z_i=\sum_j\mathsf M_{ij}\dd\lambda_j$. Its coefficient
column is $\mathsf M\boldsymbol e_i$, since $\mathsf M$ is symmetric.
Hence
\[
\langle\dd\lambda_j,\dd z_i\rangle_{M^{-1}}
=\boldsymbol e_j^{\mathsf T}\mathsf M^{-1}\mathsf M\boldsymbol e_i
=\delta_{ji}.
\]
\end{proof}

The covector inner product extends to exterior forms by
\[
\left\langle\alpha_1\wedge\cdots\wedge\alpha_k,
\beta_1\wedge\cdots\wedge\beta_k\right\rangle_{M^{-1}}
=\det\bigl(\langle\alpha_i,\beta_j\rangle_{M^{-1}}\bigr)_{i,j=1}^k.
\]
Thus the duality extends to exterior forms: for increasing $k$-tuples
$\sigma,\tau$ in $\{1,\ldots,d\}$,
\[
\langle\dd\lambda_\sigma,\dd z_\tau\rangle_{M^{-1}}
=\delta_{\sigma\tau}.
\]
The metric volume form is
\begin{equation}\label{eq:metric-volume}
\mathrm{vol}_M
=\sqrt{\det\mathsf M}\,\Omega_\lambda
=b_T^{-1/2}\Omega_\lambda.
\end{equation}
%
Set $u_i=\sqrt{\lambda_i}$. Then $\sum_i u_i^2=1$, and the metric becomes
$4\sum_i\dd u_i\otimes\dd u_i$ on the positive orthant of the unit sphere.
Consequently,
\[
\int_T\mathrm{vol}_M
=2^d\frac{|S^d|}{2^{d+1}}
=\frac{\pi^{(d+1)/2}}{\Gamma((d+1)/2)}<\infty.
\]
Although the metric coefficients diverge at the boundary, its volume
density is integrable.

\subsection{The Hodge star on bubble generators}

The Hodge star $\star_M:\Lambda^k\to\Lambda^{d-k}$ is defined by
\begin{equation}\label{eq:metric-star-definition}
\alpha\wedge\star_M\beta
=\langle\alpha,\beta\rangle_{M^{-1}}\mathrm{vol}_M.
\end{equation}
It satisfies $\star_M^2=(-1)^{k(d-k)}\operatorname{id}$.
Order $e^*$ increasingly and write $\pm_e$ for the sign of the
permutation $[e,e^*]$ of $[0,\ldots,d]$.

\begin{lemma}[Hodge star on bubble generators]\label{lm:weighted-hodge-duality}
Let $0\leq k\leq d$ and $k^*=d-k$. On $\mathring T$,
\begin{align}
\label{eq:trimmed-bubble-dual-map}
b_T^{-1/2}\star_M(b_{e^*}w_e)
&=\pm_e\dd\lambda_{e^*},&&e\in\Delta_k(T),\\
\label{eq:full-bubble-dual-map}
b_T^{-1/2}\star_M(b_e\dd\lambda_{e^*})
&=(-1)^{kk^*}\pm_e w_e,&&e\in\Delta_{k^*}(T).
\end{align}
Scalar factors commute with these maps. Moreover,
\begin{equation}\label{eq:weighted-hodge-positive}
\omega\wedge\bigl(b_T^{-1/2}\star_M\omega\bigr)
=\frac{|\omega|_{M^{-1}}^2}{b_T}\Omega_\lambda.
\end{equation}
\end{lemma}
\begin{proof}
First take $e=[0,e(1),\ldots,e(k)]$ with the last $k$ vertices increasing,
and set $\tau=e(1:k)$. For every increasing $k$-tuple $\sigma$,
\[
\begin{aligned}
\dd\lambda_\sigma\wedge(\pm_e b_T^{1/2}\dd\lambda_{e^*})
&=b_T^{1/2}\delta_{\sigma\tau}\Omega_\lambda =\langle\dd\lambda_\sigma,b_T\dd z_\tau\rangle_{M^{-1}}\mathrm{vol}_M.
\end{aligned}
\]
By \eqref{eq:whitney-log}, this verifies the defining relation \eqref{eq:metric-star-definition}
for $\star_M(b_{e^*}w_e)$ and proves \eqref{eq:trimmed-bubble-dual-map}.
Choosing any vertex of $e$ as root gives the general formula, with
$\pm_e$ accounting for orientation.
For $e\in\Delta_{k^*}(T)$, apply \eqref{eq:trimmed-bubble-dual-map} in degree $k^*$ and
apply $\star_M$ again:
\[
(-1)^{kk^*}b_{e^*}w_e
=\pm_e b_T^{1/2}\star_M(\dd\lambda_{e^*}).
\]
Multiplication by $b_e/b_T$ proves \eqref{eq:full-bubble-dual-map}.
Taking $\alpha=\beta=\omega$ in \eqref{eq:metric-star-definition} and using
\eqref{eq:metric-volume} proves \eqref{eq:weighted-hodge-positive}.
\end{proof}

The normalization $b_T^{-1/2}$ cancels the square-root factor produced
by the Hodge star, leaving polynomial forms in both identities.

\subsection{Complementary-degree duality}\label{sec:bubble-duality}

We now apply the generator identities to the geometric decompositions.
Recall that $k^*=d-k$ is the complementary form degree.

\begin{theorem}[Bubble duality]\label{thm:bubble-duality}
Let $0\leq k<d$ and set $s=r-k^*$. The normalized Hodge star gives the
isomorphisms
\begin{equation}\label{eq:bubble-dualities}
\begin{aligned}
b_T^{-1/2}\star_M:\mathbb B_{r+1}^-\Lambda^k(T)
&\longrightarrow\mathbb P_s\Lambda^{k^*}(T),
&&r\geq k^*,\\
b_T^{-1/2}\star_M:\mathbb B_r\Lambda^k(T)
&\longrightarrow\mathbb P_s^-\Lambda^{k^*}(T),
&&r\geq k^*+1.
\end{aligned}
\end{equation}
Both pairs have nondegenerate wedge pairings. In particular,
\begin{equation}\label{eq:full-bubble-dof-duality}
\left.
\begin{aligned}
&\omega\in\mathbb B_r\Lambda^k(T),\\
&\int_T\omega\wedge\eta=0
\quad\text{for every }\eta\in\mathbb P_s^-\Lambda^{k^*}(T)
\end{aligned}
\right\}
\quad\Longrightarrow\quad\omega=0.
\end{equation}
The same implication holds for $\mathbb B_{r+1}^-\Lambda^k(T)$ paired
with $\mathbb P_s\Lambda^{k^*}(T)$.
\end{theorem}

\begin{proof}
Applying \eqref{eq:trimmed-bubble-dual-map} to \eqref{eq:BT-} gives
$\mathbb P_s\Lambda^{k^*}(T)$, since the complementary forms
$\dd\lambda_{e^*}$ form a basis of constant $k^*$-forms.
For the full bubble space, \eqref{eq:full-bubble-dual-map} sends each
block of \eqref{intro:BrT} to the corresponding Whitney block. Since
$r-\ell-1=s+k^*-\ell-1$, its image is
\[
\Oplus_{\ell=k^*}^{d}\Oplus_{f\in\Delta_\ell(T)}
\Oplus_{e\in\Delta_{k^*}^0(f)}
\mathbb P_{s+k^*-\ell-1}(f)b_{f\setminus e}w_e
=\mathbb P_s^-\Lambda^{k^*}(T).
\]
The equality follows from \eqref{eq:B-} and Theorem~\ref{th:trimdec}.
Both decompositions are direct, and their scalar coefficient spaces and
form bases match, proving the isomorphisms.

Take $\eta =b_T^{-1/2}\star_M\omega$. By \eqref{eq:weighted-hodge-positive}, the integral $\int_T\omega\wedge\bigl(b_T^{-1/2}\star_M\omega\bigr) = 0$ implies $|\omega|_{M^{-1}}=0$ and consequently $\omega = 0$. 
%
\end{proof}

\begin{remark}[Earlier pointwise duality operators]
\label{rem:bubble-duality-sphere}\rm
Isaac constructed the simplex-based pointwise duality operator
$\mathring{\star}_{T,k}$~\cite[Lemma~3.7 and Theorems~3.8--3.9]{Isaac2022}.
Writing $\star_{M,k}$ for the restriction of $\star_M$ to $k$-forms,
comparison of the generator formulas, with consistent orientation and
normalization conventions, gives
\[
\mathring{\star}_{T,k}=b_T^{1/2}\star_{M,k}.
\]
Consequently, if $\mathcal D_k=b_T^{-1/2}\star_{M,k}$ denotes the
normalized Hodge star used above, then, on the corresponding polynomial
and bubble spaces,
\[
\mathcal D_k=(-1)^{k(d-k)}
\bigl(\mathring{\star}_{T,d-k}\bigr)^{-1}.
\]
Thus the maps in~\eqref{eq:bubble-dualities} realize Isaac's dualities in
reverse. The logarithmic-coordinate derivation above identifies this
operator with the Hodge star of the negative-entropy Hessian metric and
makes its action explicit on the rooted Whitney--bubble generators.

Let $\Phi:S^d_+\to\mathring T$ be given by $\lambda_i\circ\Phi=u_i^2$,
where $S^d_+=\{u_i>0:\sum_i u_i^2=1\}$. Since $\Phi^*M=4g_{S^d}$,
Berchenko--Kogan's map on $q$-forms~\cite[Corollary~3.3]{BerchenkoKogan2021}
is, with consistent orientations,
\[
B_q=(\Phi^*)^{-1}\circ\bigl((\textstyle\prod_i u_i)\star_{S^d}\bigr)\circ\Phi^*
=2^{2q-d}b_T^{1/2}\star_M.
\]
Thus
\[
B_{d-k}^{-1}=(-1)^{k(d-k)}2^{2k-d}b_T^{-1/2}\star_M.
\]
The dualities above are the same maps in reverse, up to this constant
and sign. Our generator formulas give a direct barycentric description
on each geometric face block. The positivity proof is the simplex
version of~\cite[Corollary~3.4]{BerchenkoKogan2021}.
\end{remark}

\section{A nodal Whitney--Lagrange basis for the trimmed family}
\label{sec:nodal-whitney-lagrange}

The preceding duality gives integral degrees of freedom for bubble spaces. We now construct a basis dual to tangential point evaluations. The same geometric decomposition supplies the face indices, while shifted Lagrange factors enforce cardinality across all nodes.

\subsection{A common simplicial lattice}
Using the notation of Section~\ref{sec:simplicial-lattices}, we place the nodal sets for $\mathbb P_{r+k-\ell}(f)$ on the common lattice $\mathbb T_m^d$, where
\begin{equation}\label{eq:m-common}
m:=r+k+1.
\end{equation}
The degree $m$ is independent of the face dimension for fixed $r$ and $k$.



For $f\in\Delta_\ell(T)$, let $\alpha_f=(\alpha_{f(0)},\ldots,\alpha_{f(\ell)})$ denote the local multi-index. The Bernstein representation is
\[
\mathbb P_m(f)=\operatorname{span}\{\lambda_f^{\alpha_f}:\alpha_f\in\mathbb T_m^\ell\}.
\]
For $\alpha\in\mathbb T_m^\ell(\mathring f)$, every entry of $\alpha_f$ is positive. Subtracting $1$ componentwise gives the bijection
\[
\alpha_f\longmapsto\beta_f:=\alpha_f-1\in\mathbb T_{m_f}^\ell,
\qquad m_f:=m-(\ell+1)=r+k-\ell.
\]
The corresponding monomials factor as
\begin{equation}\label{eq:lambdaalphabeta}
\lambda_f^{\alpha_f}=b_f\lambda_f^{\beta_f}.
\end{equation}
Thus division by $b_f$ identifies the scalar bubble space spanned by the interior Bernstein monomials with $\mathbb P_{m_f}(f)$. 
If $m_f<0$, the interior lattice is empty and the bubble space is zero. 

Let $\{\phi_\alpha:\alpha\in\mathbb T_m^d\}$ be the Lagrange basis of
$\mathbb P_m(T)$ associated with the lattice $\mathbb T_m^d$, i.e.,
\[
\phi_\alpha(x_\beta)=\delta_{\alpha\beta}, \qquad \alpha, \beta \in \mathbb T_m^d.
\]
Explicitly,
\[
\phi_\alpha(x)
=\prod_{i=0}^d\prod_{j=0}^{\alpha_i-1}
\frac{m\lambda_i(x)-j}{\alpha_i-j},
\]
with empty products equal to $1$.
For $\alpha\in\mathbb T_m^\ell(\mathring f)$, the $j=0$ factors in the product formula show that
%
$\phi_\alpha$ is divisible
by $b_f$.
Define the shifted Lagrange polynomial
\begin{equation}\label{eq:shifted-lagrange}
\phi_{\alpha,f}
:=
b_f(x_\alpha)\frac{\phi_\alpha}{b_f},
\qquad
\phi_\alpha
=
\frac{b_f}{b_f(x_\alpha)}\phi_{\alpha,f}.
\end{equation}
Equivalently, the standard product formula for the simplex Lagrange
basis gives the explicit shifted formula
\[
\phi_{\alpha,f}
=
\prod_{i\in f}
\prod_{j=1}^{\alpha_i-1}
\frac{m\lambda_i-j}{\alpha_i-j},
\]
where an empty product equals $1$. Since $\operatorname{supp}\alpha=f$, the factors with $i\notin f$ in the unshifted formula are already empty. The shift removes the $j=0$ factor $m\lambda_i/\alpha_i$ for each $i\in f$, so the products over $j$ now begin at $j=1$.

\begin{lemma}
Let $f\in\Delta_\ell(T)$ and $\alpha\in\mathbb T_m^\ell(\mathring f)$. For every $g\in\Delta_{\ell'}(T)$ containing $f$ and every $\gamma\in\mathbb T_m^{\ell'}(\mathring g)$,
\begin{equation}\label{eq:shifted-cardinality}
\phi_{\alpha,f}(x_\gamma)=\delta_{\alpha\gamma}.
\end{equation}
Consequently,
\[
\left\{\phi_{\alpha,f}|_f:\alpha\in\mathbb T_m^\ell(\mathring f)\right\}
\]
is the cardinal basis of $\mathbb P_{m_f}(f)$ at the interior lattice nodes.
\end{lemma}

\begin{proof}
Since $\operatorname{supp}\gamma=g$ and $f\subseteq g$,
\[
b_f(x_\gamma)=\prod_{i\in f}\frac{\gamma_i}{m}>0.
\]
Evaluating \eqref{eq:shifted-lagrange} at $x_\gamma$ gives
\[
\delta_{\alpha\gamma}=\phi_\alpha(x_\gamma)=\frac{b_f(x_\gamma)}{b_f(x_\alpha)}\phi_{\alpha,f}(x_\gamma).
\]
The prefactor is nonzero and equals one when $\gamma=\alpha$, which proves \eqref{eq:shifted-cardinality}. Taking $g=f$ gives cardinality on the interior nodes of $f$. Since each restriction belongs to $\mathbb P_{m_f}(f)$ and the number of nodes equals its dimension, these restrictions form a basis.
\end{proof}

The shifted extension satisfies \eqref{eq:shifted-cardinality} at nodes on every face containing $f$. On $T$, it generally differs from the fixed-degree Bernstein extension in Section~\ref{sec:bernstein-extension}. Its values on faces not containing $f$ are unrestricted; the trace properties of the Whitney--bubble factor will handle those faces.

\begin{example}[Shifted Lagrange versus Bernstein extension]\rm
Let $f=[0,1]$, $m=3$, and $\alpha=(2,1,0,\ldots,0)$. The scalar Lagrange function and its shifted quotient are
\[
\phi_\alpha=\frac{9}{2}\lambda_0\lambda_1(3\lambda_0-1),
\qquad b_f(x_\alpha)=\frac{2}{9},
\qquad \phi_{\alpha,f}=b_f(x_\alpha)\frac{\phi_\alpha}{b_f}=3\lambda_0-1.
\]
Here $m_f=1$. Since $\lambda_0+\lambda_1=1$ on $f$, the restriction of $\phi_{\alpha,f}$ is $2\lambda_0-\lambda_1$. Its degree-one Bernstein extension is therefore $2\lambda_0-\lambda_1$, whereas
\[
\phi_{\alpha,f}=(2\lambda_0-\lambda_1)-\sum_{i\notin f}\lambda_i.
\]
Both extensions agree on $f$. The extra term makes the shifted extension vanish at the interior lattice nodes of larger faces containing $f$.
\end{example}

\subsection{Pointwise duality of rooted Whitney forms}

For $e\in\Delta_k^0(f)$, we have $e(0)=f(0)$. With the other vertices in increasing order, we orient $e$ as
\[
e=[f(0),e(1),\ldots,e(k)].
\]

Recall that a $k$-form on $f$ is naturally an element of
\[
\alt^k(\mathscr T^f)
=
\left(\bigwedge^k\mathscr T^f\right)^*.
\]
For vectors $u_1,\ldots,u_k\in\mathscr T^f$, we use the canonical pairing
\begin{equation}\label{eq:kform-kvector-pairing}
\omega(u_1\wedge\cdots\wedge u_k)
:=
\omega(u_1,\ldots,u_k),
\qquad
\omega\in\alt^k(\mathscr T^f).
\end{equation}
Thus no identification between vectors and covectors is needed.

\begin{lemma}[Pointwise Whitney duality]
\label{lm:whitney-point-duality}
Let $f\in\Delta_\ell(T)$, let $x_\alpha$ be any point of
$\mathring f$, and
$e\in\Delta_k^0(f)$. Define
\begin{equation}\label{eq:whitney-dual-vector}
\tau_{\alpha,e}
:=
(\texttt{v}_{e(1)}-x_\alpha)\wedge\cdots\wedge
(\texttt{v}_{e(k)}-x_\alpha)
\in\bigwedge^k\mathscr T^f,
\end{equation}
with $\tau_{\alpha,e}:=1$ for $k=0$. Then, for
$e,e'\in\Delta_k^0(f)$,
\begin{equation}\label{eq:whitney-point-duality}
w_{e'}^f(x_\alpha)(\tau_{\alpha,e})
=
\lambda_{f(0)}(x_\alpha)\delta_{ee'}.
\end{equation}
\end{lemma}

\begin{proof}
Write $e'=[f(0),e'(1),\ldots,e'(k)]$ and set
$u_s=\texttt v_{e(s)}-x_\alpha$. The affine coordinates satisfy
\[
\dd_f\lambda_j(u_s)=\delta_{j,e(s)}-\lambda_j(x_\alpha).
\]
The Whitney pairing is the determinant whose first row is
$(\lambda_j(x_\alpha))_{j\in e'}$ and whose row $s+1$ is
$(\dd_f\lambda_j(u_s))_{j\in e'}$. Add the first row to each later row.
Since no $e(s)$ is the root, expansion in the first column gives
\[
w_{e'}^f(x_\alpha)(\tau_{\alpha,e})
=\lambda_{f(0)}(x_\alpha)
\det\bigl(\delta_{e'(t),e(s)}\bigr)_{s,t=1}^k
=\lambda_{f(0)}(x_\alpha)\delta_{ee'}.
\]
Here the determinant is $1$ for the common ordered vertex set and
$0$ otherwise. For $k=0$, the empty determinant is $1$.
\end{proof}

\subsection{Closed-form cardinal basis}
We combine the shifted scalar cardinal factors with the pointwise Whitney duality to obtain a basis dual to tangential evaluations.
\begin{theorem}[Nodal Whitney--Lagrange basis]
\label{thm:nodal-whitney-lagrange}
Assume $r\geq0$ and $0\leq k\leq d$, and let
\[
m:=r+k+1.
\]
For $k\leq\ell\leq\min\{d,r+k\}$,
\[
f\in\Delta_\ell(T),\qquad
\alpha\in\mathbb T_m^\ell(\mathring f),\qquad
e\in\Delta_k^0(f),
\]
define
\begin{equation}\label{eq:nodal-trimmed-basis}
\Phi_{\alpha,e}^{f,T}
:=
\frac{
\phi_{\alpha,f}\,b_{f\setminus e}\,w_e
}{
\lambda_{f(0)}(x_\alpha)\,
b_{f\setminus e}(x_\alpha)
}.
\end{equation}
The denominator is positive because $x_\alpha\in\mathring f$, and hence
$\lambda_i(x_\alpha)>0$ for every $i\in f$.
For $k\leq\ell'\leq\min\{d,r+k\}$,
\[
g\in\Delta_{\ell'}(T),\qquad
\gamma\in\mathbb T_m^{\ell'}(\mathring g),\qquad
e'\in\Delta_k^0(g),
\]
define the face-local functional
\begin{equation}\label{eq:nodal-trimmed-dof}
\begin{aligned}
N_{\gamma,e'}^g
&\in
\bigl(\mathbb P_{r+1}^-\Lambda^k(T)\bigr)^*,\\
N_{\gamma,e'}^g(\omega)
&:=
(\tr_g\omega)(x_\gamma)
\bigl(\tau_{\gamma,e'}\bigr).
\end{aligned}
\end{equation}
Then
\begin{equation}\label{eq:global-biorthogonality}
N_{\gamma,e'}^g
\bigl(\Phi_{\alpha,e}^{f,T}\bigr)
=
\delta_{fg}\delta_{\alpha\gamma}\delta_{ee'}.
\end{equation}
Consequently,
\begin{equation}\label{eq:nodal-trimmed-basis-set}
\left\{
\Phi_{\alpha,e}^{f,T}:
\begin{array}{l}
f\in\Delta_\ell(T),\quad
k\leq\ell\leq\min\{d,r+k\},\\
\alpha\in\mathbb T_m^\ell(\mathring f),\quad
e\in\Delta_k^0(f)
\end{array}
\right\}
\end{equation}
is a basis of $\mathbb P_{r+1}^-\Lambda^k(T)$, and the functionals
\eqref{eq:nodal-trimmed-dof} form its dual basis.
\end{theorem}

\begin{proof}
By~\eqref{eq:shifted-lagrange}, the scalar factor satisfies
\[
\deg\bigl(\phi_{\alpha,f}b_{f\setminus e}\bigr)
\leq(r+k-\ell)+(\ell-k)=r.
\]
Lemma~\ref{lm:qwf} therefore gives
\[
\Phi_{\alpha,e}^{f,T}
\in\mathbb P_{r+1}^-\Lambda^k(T).
\]

We prove~\eqref{eq:global-biorthogonality}. If $f\nsubseteq g$, then
Lemma~\ref{lem:facebubbletrace} gives
\[
\tr_g\bigl(b_{f\setminus e}w_e\bigr)=0, \text{ and therefore } N_{\gamma,e'}^g
\bigl(\Phi_{\alpha,e}^{f,T}\bigr)=0.
\]

Now suppose $f\subseteq g$. By~\eqref{eq:shifted-cardinality},
\[
\phi_{\alpha,f}(x_\gamma)
=
\delta_{\alpha\gamma}.
\]
Thus the pairing vanishes unless $\alpha=\gamma$. Since
$\operatorname{supp}\alpha=f$ and $\operatorname{supp}\gamma=g$, the
equality $\alpha=\gamma$ implies $f=g$. In this remaining case,
$\tr_f w_e=w_e^f$, and Lemma~\ref{lm:whitney-point-duality} yields
\[
\begin{aligned}
N_{\alpha,e'}^f
\bigl(\Phi_{\alpha,e}^{f,T}\bigr)
&=
\frac{
b_{f\setminus e}(x_\alpha)
w_e^f(x_\alpha)(\tau_{\alpha,e'})
}{
\lambda_{f(0)}(x_\alpha)
b_{f\setminus e}(x_\alpha)
} =
\delta_{ee'}.
\end{aligned}
\]
This proves~\eqref{eq:global-biorthogonality}.

For each $f\in\Delta_\ell(T)$,
\[
\#\mathbb T_m^\ell(\mathring f)
=
\binom{r+k}{\ell},
\qquad
\#\Delta_k^0(f)
=
\binom{\ell}{k}.
\]
Hence the number of basis functions associated with $f$ is
\[
\binom{r+k}{\ell}\binom{\ell}{k}
=
\dim\mathbb B_{r+1}^-\Lambda^k(T;f).
\]
Summing over all faces and using Theorem~\ref{th:trimdec} completes the proof.
\end{proof}

Although the cardinal functions use the same face indices as the
geometric decomposition, one generally has
\[
\Phi_{\alpha,e}^{f,T}
\notin
\mathbb B_{r+1}^-\Lambda^k(T;f).
\]
The face $f$ records where the nodal data lie. The shifted extension
adds corrections on larger faces containing $f$.

\begin{remark}[Scalar reduction]\rm
For $k=0$, the only rooted $0$-face of $f$ is
$e=\{f(0)\}$. Since $w_e=\lambda_{f(0)}$ and
$b_{f\setminus e}w_e=b_f$, equations~\eqref{eq:shifted-lagrange}
and~\eqref{eq:nodal-trimmed-basis} give
\[
\Phi_{\alpha,e}^{f,T}
=
\frac{\phi_{\alpha,f}b_f}{b_f(x_\alpha)}
=
\phi_\alpha.
\]
Moreover, $\tau_{\alpha,e}=1$ and
$N_{\alpha,e}^f(u)=u(x_\alpha)$. Since the degree-$(r+1)$ lattice is
the disjoint union of its nodes according to their support faces, the
basis and degrees of freedom reduce exactly to the standard
degree-$(r+1)$ Lagrange element.
\end{remark}

\begin{remark}[Scalar product formula and earlier vector bases]
\label{rem:product-comparison}\rm
Substituting~\eqref{eq:shifted-lagrange} into
\eqref{eq:nodal-trimmed-basis} and cancelling $b_f=b_e b_{f\setminus e}$ gives
\begin{equation}\label{eq:compact-cardinal}
\Phi_{\alpha,e}^{f,T}
=\left(\prod_{i\in e\setminus\{f(0)\}}\lambda_i(x_\alpha)\right)
\frac{\phi_\alpha}{b_e}\,w_e.
\end{equation}
Here $\phi_\alpha$ is the scalar degree-$m$ Lagrange function and the
quotient is a polynomial of degree at most $r$. Division means polynomial cancellation before evaluation, including at boundary points where $b_e=0$. The factored formula~\eqref{eq:nodal-trimmed-basis} avoids this division. For $d=3$, $k=1$,
and $e=[0, 1]$ rooted at $0$, this becomes
\[
\Phi_{\alpha,e}^{f,T}
=\lambda_1(x_\alpha)\frac{\phi_\alpha}{\lambda_0\lambda_1}w_{01}.
\]
The tetrahedral product bases of~\cite{GragliaWiltonPeterson1997},
as written in~\cite[(9)--(13)]{HooghiemstraHeulVuik2010},
have this same scalar--Whitney factor after cancellation, with a different
normalization. Rooted edges select independent generators at each node;
different component conventions are related by local invertible changes.
The present formula treats all form degrees by the same construction.
\end{remark}

\subsection{Transfer from other scalar node sets}

The division in~\eqref{eq:compact-cardinal} does not require equally
spaced nodes. It suffices that the scalar nodes are unisolvent on every face.

\begin{theorem}[Transfer from scalar interpolation]\label{thm:scalar-node-transfer}
Let $r\geq0$, $0\leq k\leq d$, and $m=r+k+1$.
Let $X\subset T$ be a finite set such that, for every nonempty face $F$
including $T$, evaluation is an isomorphism
\[
\mathbb P_m(F)\longrightarrow\mathbb R^{X\cap F}.
\]
Let $\{\phi_x:x\in X\}$ be the scalar cardinal basis on $T$.
For $x\in X\cap\mathring f$, with $\dim f\geq k$, choose a root in $f$ and label it $0$ locally.
For each $k$-face $e=[0,e(1),\ldots,e(k)]\subseteq f$, set
\begin{equation}\label{eq:scalar-node-transfer}
\begin{aligned}
\Phi_{x,e}
&=\left(\prod_{j=1}^k\lambda_{e(j)}(x)\right)
\frac{\phi_x}{b_e}\,w_e,\\
N_{x,e}(\omega)
&=(\tr_f\omega)(x)
\bigl((\texttt v_{e(1)}-x)\wedge\cdots\wedge(\texttt v_{e(k)}-x)\bigr).
\end{aligned}
\end{equation}
The quotient is a polynomial. These functions form a basis of
$\mathbb P_{r+1}^-\Lambda^k(T)$ dual to the displayed functionals.
The root and its local labeling may be chosen separately at each node. Empty products and
wedges equal $1$.
\end{theorem}

\begin{proof}
Assume first $d\geq1$. If $x\in\mathring f$ and $i\in f$, then
$x\notin F_i$. The scalar cardinal function $\phi_x$ vanishes on
$X\cap F_i$, so face unisolvence implies $\phi_x|_{F_i}=0$.
Hence $\lambda_i$ divides $\phi_x$. These different linear factors are
relatively prime, so $b_f$ divides $\phi_x$. Since $e\subseteq f$,
$\phi_x/b_e$ is a polynomial of degree at most $m-k-1=r$.
Lemma~\ref{lm:qwf} proves membership in the trimmed space.

Pair $\Phi_{x,e}$ with a functional at $y\in X\cap\mathring g$.
If $f\nsubseteq g$, choose $i\in f\setminus g$.
When $i\in f\setminus e$, the quotient retains the factor $\lambda_i$;
when $i\in e$, the Whitney form has zero trace on $g$.
Thus the pairing is zero. If $f\subseteq g$, then $b_e(y)>0$,
and scalar cardinality makes the pairing zero unless $x=y$.
In that case $f=g$, and Lemma~\ref{lm:whitney-point-duality}, with root
$0$, gives
\[
w_e^f(x)\bigl((\texttt v_{e'(1)}-x)\wedge\cdots\wedge
(\texttt v_{e'(k)}-x)\bigr)=\lambda_0(x)\delta_{ee'}.
\]
Since $b_e(x)=\lambda_0(x)\prod_{j=1}^k\lambda_{e(j)}(x)$, the
normalization in~\eqref{eq:scalar-node-transfer} gives $\delta_{ee'}$.

The facewise isomorphisms imply $\#(X\cap F)=\binom{m+\dim F}{\dim F}$.
Subtracting the counts on proper faces recursively gives
$\#(X\cap\mathring f)=\binom{m-1}{\dim f}$: the principal lattice
has the same closed-face counts and hence the same interior counts.
Each such node supplies $\binom{\dim f}{k}$ functions.
The dimension identity in the proof of Theorem~\ref{th:trimdec}
therefore completes the argument. If $d=0$, necessarily $k=0$,
$X=T$, and $\phi_x=\lambda_0=1$, so the assertion is immediate.
\end{proof}

Only nodes on faces of dimension at least $k$ are used for the form
element. Unisolvence on $T$ alone is insufficient: the face hypotheses
ensure the polynomial divisions. For nonuniform nodes, the scalar
cardinal functions may require a separate construction or linear solve.
The theorem requires no additional form-valued solve and makes no claim about the interpolation norm. On a mesh, the scalar node sets on shared faces
must agree. For componentwise identification of the degrees of freedom,
a common root and orientation convention may be fixed on each shared
face; different rooted systems are otherwise related by an invertible
local change of components.

\subsection{Interpolation and conformity}

Let $\mathcal T_h$ be a conforming simplicial triangulation of a
polyhedral Lipschitz domain $\Omega$.
Recall that $H\Lambda^k(\Omega)$ consists of square-integrable $k$-forms whose weak exterior derivative is square-integrable.
The functional $N_{\alpha,e}^f$ depends only on the oriented subsimplex $f$,
the point $x_\alpha$, and the rooted face $e$, not on the ambient element.
After fixing orientations and vertex orderings globally, neighboring
elements use identical functionals on a common face. For
$0\leq k\leq d-1$, let $F$ be a common facet.
Theorem~\ref{thm:nodal-whitney-lagrange}, applied intrinsically on $F$,
shows that the functionals indexed by subfaces $f\subseteq F$ are
unisolvent for $\mathbb P_{r+1}^-\Lambda^k(F)$. By~\eqref{eq:trimmed-face-trace}, both element traces lie in this space. Hence, equality of the
shared values is equivalent to equality of the tangential traces on
$F$. For $k=d$, no interelement trace matching is required. Let $V_h^k$ be the assembled piecewise trimmed space with matching shared degrees of freedom. The standard characterization of $H\Lambda^k$ conformity by matching facet traces~\cite[Lemma~2.2]{ArnoldFalkWinther2009} gives
\[
V_h^k
=
\{\omega_h\in H\Lambda^k(\Omega):
\omega_h|_T\in\mathbb P_{r+1}^-\Lambda^k(T)
\text{ for all }T\in\mathcal T_h\}.
\]
Point evaluation is generally not bounded on $H\Lambda^k$. For
$\omega\in C^0\Lambda^k(\overline T)$, define the nodal interpolation
operator
\begin{equation}\label{eq:nodal-interpolant}
I_{r+1,T}^{-,k}\omega
:=
\sum_{\ell=k}^{\min\{d,r+k\}}
\sum_{f\in\Delta_\ell(T)}
\sum_{\alpha\in\mathbb T_m^\ell(\mathring f)}
\sum_{e\in\Delta_k^0(f)}
N_{\alpha,e}^f(\omega)\Phi_{\alpha,e}^{f,T}.
\end{equation}
By~\eqref{eq:global-biorthogonality},
\[
I_{r+1,T}^{-,k}p=p,
\qquad
p\in\mathbb P_{r+1}^-\Lambda^k(T).
\]
Thus $I_{r+1,T}^{-,k}$ is a polynomial projector on continuous forms.
Its compatibility with face restriction is established next; the exterior
derivative is discussed separately below.

\begin{proposition}[Root independence and face restriction]
\label{prop:root-trace}
At each active node $x_\alpha\in\mathring f$, the interpolation
conditions are equivalent to
\[
(\tr_f I_{r+1,T}^{-,k}\omega)(x_\alpha)
=(\tr_f\omega)(x_\alpha)
\quad\text{in }\operatorname{Alt}^k(\mathscr T^f).
\]
Thus the interpolant is independent of the roots and vertex orderings.
For every face $F$ of dimension at least $k$,
\begin{equation}\label{eq:interpolant-trace}
\tr_F I_{r+1,T}^{-,k}\omega
=I_{r+1,F}^{-,k}(\tr_F\omega).
\end{equation}
The same statements hold for the interpolant in
Theorem~\ref{thm:scalar-node-transfer}, using $X\cap F$ on $F$.
\end{proposition}
\begin{proof}
By Lemma~\ref{lm:whitney-point-duality} and the dimension count
$\dim\bigwedge^k\mathscr T^f=\binom{\dim f}{k}$, the rooted tangent
multivectors form a basis of $\bigwedge^k\mathscr T^f$ at the node.
Matching all scalar functionals is therefore exactly matching the whole
intrinsic form value. These conditions do not involve a root or ordering,
so uniqueness proves independence of those choices. By~\eqref{eq:trimmed-face-trace}, the trace on $F$
belongs to $\mathbb P_{r+1}^-\Lambda^k(F)$; it satisfies the intrinsic
nodal conditions. Unisolvence on $F$ proves~\eqref{eq:interpolant-trace}.
\end{proof}

\begin{remark}[Affine covariance]\rm
Let $F_T:\widehat T\to T$ be an affine bijection with compatible vertex
labels, and let $\widehat f$ and $\widehat e$ denote the corresponding
labeled subsimplices of $\widehat T$.
Barycentric coordinates, Whitney forms, lattice points, and the tangent
multivectors in \eqref{eq:whitney-dual-vector} transform naturally under
$F_T$.  Consequently,
\[
F_T^*\Phi_{\alpha,e}^{f,T}
=
\Phi_{\alpha,\widehat e}^{\widehat f,\widehat T},
\qquad
N_{\alpha,e}^{f}(\omega)
=
N_{\alpha,\widehat e}^{\widehat f}(F_T^*\omega).
\]
Thus the element can be generated from a reference simplex.
\end{remark}

Root independence concerns the interpolant. Individual basis functions
may mix under a change of root; they need not merely permute or change
sign. Invariant bases under vertex permutations are studied in
\cite{LichtSymmetry2024}.

\begin{remark}[The exterior derivative]\rm
Equation~\eqref{eq:interpolant-trace} is a restriction identity.
The interpolant does not in general commute with $\dd$.
On $[0,1]$, take $r=0$ and $u(x)=x^3$. Scalar interpolation uses the
endpoints and one-form interpolation uses the midpoint, so
\[
\dd I_1^{-,0}u=\dd x,\qquad
I_1^{-,1}(\dd u)=\tfrac34\,\dd x.
\]
The assembled space remains the standard conforming trimmed space.
\end{remark}

%
%
%

\subsection{Examples}
\label{sec:examples}

For $r=0$, only $k$-faces are active. Each carries one node, its barycenter $x_f$, and the basis reduces to the rescaled Whitney forms
\begin{equation}\label{eq:r0-basis}
\Phi_f=(k+1)w_f,\qquad x_f=\frac{1}{k+1}\sum_{i\in f}\texttt v_i.
\end{equation}
Figure~\ref{fig:r0-nodal-lattice} shows these nodes on a tetrahedron. Figure~\ref{fig:k1-nodal-lattice} shows the higher-order cases tabulated below.

\begin{figure}[htbp]
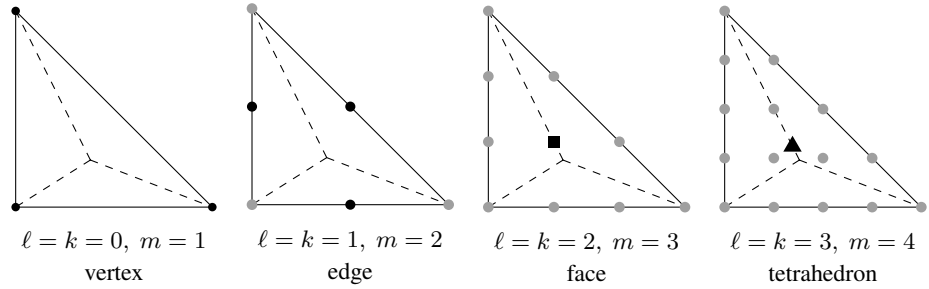

\centering
\begin{minipage}{0.24\textwidth}
\centering
\TetLattice{1}{0}\\[1mm]
{\small $\ell=k=0,\ m=1$\\
vertex}
\end{minipage}
\begin{minipage}{0.24\textwidth}
\centering
\TetLattice{2}{1}\\[1mm]
{\small $\ell=k=1,\ m=2$\\
edge}
\end{minipage}
\begin{minipage}{0.24\textwidth}
\centering
\TetLattice{3}{2}\\[1mm]
{\small $\ell=k=2,\ m=3$\\
face}
\end{minipage}
\begin{minipage}{0.24\textwidth}
\centering
\TetLattice{4}{3}\\[1mm]
{\small $\ell=k=3,\ m=4$\\
tetrahedron}
\end{minipage}
\caption{Lowest-order elements on a tetrahedron. Black markers indicate DoF locations; gray markers indicate auxiliary points of the common simplicial lattice. Squares mark face points, and triangles mark interior points. For clarity, only points on the front face and in the tetrahedron interior are shown.}
\label{fig:r0-nodal-lattice}
\end{figure}

\begin{figure}[htbp]
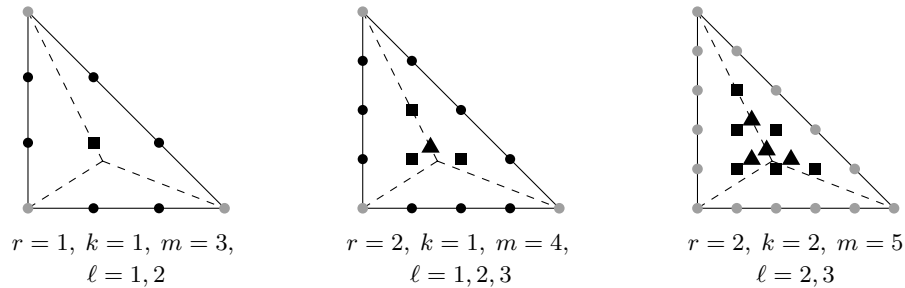

\centering
\begin{minipage}{0.30\textwidth}
\centering
\TetLattice{3}{1}\\[1mm]
{\small $r=1,\ k=1,\ m=3,\ $\\
$\ell=1,2$}
\end{minipage}
\hfill
\begin{minipage}{0.30\textwidth}
\centering
\TetLattice{4}{1}\\[1mm]
{\small $r=2,\ k=1,\ m=4,\ $\\
$\ell=1,2,3$}
\end{minipage}
\hfill
\begin{minipage}{0.30\textwidth}
\centering
\TetLattice{5}{2}\\[1mm]
{\small $r=2, \ k=2, \ m = 5$\\
$\ell = 2, 3$}
\end{minipage}
\caption{Nodal lattices on a tetrahedron for $(k,r)=(1,1)$, $(1,2)$, and $(2,2)$, from left to right. Each active node on an $\ell$-face carries $\binom{\ell}{k}$ DoFs. Marker conventions are as in Figure~\ref{fig:r0-nodal-lattice}.}
\label{fig:k1-nodal-lattice}
\end{figure}

Tables~\ref{tab:cardinal-triangles} and \ref{tab:cardinal-volume} collect three-dimensional one-form and two-form bases, respectively, with separate blocks for $r=1,2$. Each block includes all active face dimensions and the total dimension of the element space.

Each row uses a representative face $f=[0,\ldots,\ell]$ rooted at $0$. The entries of $\alpha$ follow this vertex order, with indices outside $f$ equal to zero and $\lambda_i(x_\alpha)=\alpha_i/m$, where $m=r+k+1$. The tables use $T=[0,1,2,3]$. Other faces follow by relabeling with the corresponding orientation. We abbreviate $w_{[i,j,\ldots]}$ as $w_{ij\cdots}$. Parentheses list separate basis functions at the same node; a scalar factor outside multiplies every entry.

\begin{table}[!htbp]
\centering
\small
\setlength{\tabcolsep}{4pt}
\renewcommand{\arraystretch}{1.35}
\resizebox{\textwidth}{!}{%
\begin{tabular}{cccll}
\toprule
Parameters & $f$ & $\alpha$ & $\phi_{\alpha,f}$ & Basis functions $\Phi_{\alpha,e}^{f,T}$\\
\midrule
\multirow{3}{*}{\shortstack{$d=3$, $k=1$\\$r=1$, $m=3$\\$\dim=20$}} & $[0,1]$ & $(2,1)$ & $(3\lambda_0-1)$ & $\frac{3}{2}\phi_{\alpha,f}w_{01}$\\
 & $[0,1]$ & $(1,2)$ & $(3\lambda_1-1)$ & $3\phi_{\alpha,f}w_{01}$\\
 & $[0,1,2]$ & $(1,1,1)$ & $1$ & $(9\lambda_2w_{01},9\lambda_1w_{02})$\\
\midrule
\multirow{7}{*}{\shortstack{$d=3$, $k=1$\\$r=2$, $m=4$\\$\dim=45$}} & $[0,1]$ & $(3,1)$ & $\frac{1}{2}(4\lambda_0-1)(4\lambda_0-2)$ & $\frac{4}{3}\phi_{\alpha,f}w_{01}$\\
 & $[0,1]$ & $(2,2)$ & $(4\lambda_0-1)(4\lambda_1-1)$ & $2\phi_{\alpha,f}w_{01}$\\
 & $[0,1]$ & $(1,3)$ & $\frac{1}{2}(4\lambda_1-1)(4\lambda_1-2)$ & $4\phi_{\alpha,f}w_{01}$\\
 & $[0,1,2]$ & $(2,1,1)$ & $(4\lambda_0-1)$ & $\phi_{\alpha,f}(8\lambda_2w_{01},8\lambda_1w_{02})$\\
 & $[0,1,2]$ & $(1,2,1)$ & $(4\lambda_1-1)$ & $\phi_{\alpha,f}(16\lambda_2w_{01},8\lambda_1w_{02})$\\
 & $[0,1,2]$ & $(1,1,2)$ & $(4\lambda_2-1)$ & $\phi_{\alpha,f}(8\lambda_2w_{01},16\lambda_1w_{02})$\\
 & $T$ & $(1,1,1,1)$ & $1$ & $(64\lambda_2\lambda_3w_{01},64\lambda_1\lambda_3w_{02},64\lambda_1\lambda_2w_{03})$\\
\bottomrule
\end{tabular}}
\medskip
\caption{Three-dimensional one-form bases ($k=1$) for $r=1,2$. Each block includes all active face dimensions, with one representative edge and triangular face. For $r=1$, there are no tetrahedron-interior nodes.}
\label{tab:cardinal-triangles}
\end{table}

For implementation, \eqref{eq:nodal-trimmed-basis} reads
\[
\Phi_{\alpha,e}^{f,T}=\frac{m^{\ell-k+1}}{\alpha_{f(0)}\prod_{i\in f\setminus e}\alpha_i}\,\phi_{\alpha,f}b_{f\setminus e}w_e,
\qquad
\phi_{\alpha,f}=\prod_{i\in f}\frac{\prod_{j=1}^{\alpha_i-1}(m\lambda_i-j)}{(\alpha_i-1)!}.
\]
The tables display $\phi_{\alpha,f}$ separately to avoid repeating its factors. Each basis function is dual to the tangential evaluation in \eqref{eq:nodal-trimmed-dof}.

\begin{remark}[Reading the tables for two-dimensional elements]\rm
For fixed $r$ and $k$, the formulas associated with a face are independent of the ambient dimension: their tangential traces give the intrinsic face bases. To obtain the two-dimensional one-form basis from Table~\ref{tab:cardinal-triangles}, retain the rows for $[0,1]$ and $[0,1,2]$, interpret $[0,1,2]$ as the triangle $T$, and omit the tetrahedron-interior rows. Use the barycentric coordinates and Whitney forms of the triangle, and relabel the edge rows for its three edges with consistent orientations. This gives $3\cdot2+2=8$ functions for $r=1$ and $3\cdot3+3\cdot2=15$ for $r=2$. The parameters $k,r,m$ are unchanged, while $d=2$ and the total dimensions replace those in the table. Similarly, the triangular-face rows of Table~\ref{tab:cardinal-volume} give the two-dimensional top-form bases, with $3$ and $6$ functions for $r=1$ and $r=2$, respectively.
\end{remark}

\begin{table}[!htbp]
\centering
\small
\setlength{\tabcolsep}{4pt}
\renewcommand{\arraystretch}{1.35}
\resizebox{\textwidth}{!}{%
\begin{tabular}{cccll}
\toprule
Parameters & $f$ & $\alpha$ & $\phi_{\alpha,f}$ & Basis functions $\Phi_{\alpha,e}^{f,T}$\\
\midrule
\multirow{4}{*}{\shortstack{$d=3$, $k=2$\\$r=1$, $m=4$\\$\dim=15$}} & $[0,1,2]$ & $(2,1,1)$ & $(4\lambda_0-1)$ & $2\phi_{\alpha,f}w_{012}$\\
 & $[0,1,2]$ & $(1,2,1)$ & $(4\lambda_1-1)$ & $4\phi_{\alpha,f}w_{012}$\\
 & $[0,1,2]$ & $(1,1,2)$ & $(4\lambda_2-1)$ & $4\phi_{\alpha,f}w_{012}$\\
 & $T$ & $(1,1,1,1)$ & $1$ & $(16\lambda_1w_{023},16\lambda_2w_{013},16\lambda_3w_{012})$\\
\midrule
\multirow{10}{*}{\shortstack{$d=3$, $k=2$\\$r=2$, $m=5$\\$\dim=36$}} & $[0,1,2]$ & $(3,1,1)$ & $\frac{1}{2}(5\lambda_0-1)(5\lambda_0-2)$ & $\frac{5}{3}\phi_{\alpha,f}w_{012}$\\
 & $[0,1,2]$ & $(1,3,1)$ & $\frac{1}{2}(5\lambda_1-1)(5\lambda_1-2)$ & $5\phi_{\alpha,f}w_{012}$\\
 & $[0,1,2]$ & $(1,1,3)$ & $\frac{1}{2}(5\lambda_2-1)(5\lambda_2-2)$ & $5\phi_{\alpha,f}w_{012}$\\
 & $[0,1,2]$ & $(2,2,1)$ & $(5\lambda_0-1)(5\lambda_1-1)$ & $\frac{5}{2}\phi_{\alpha,f}w_{012}$\\
 & $[0,1,2]$ & $(2,1,2)$ & $(5\lambda_0-1)(5\lambda_2-1)$ & $\frac{5}{2}\phi_{\alpha,f}w_{012}$\\
 & $[0,1,2]$ & $(1,2,2)$ & $(5\lambda_1-1)(5\lambda_2-1)$ & $5\phi_{\alpha,f}w_{012}$\\
 & $T$ & $(2,1,1,1)$ & $(5\lambda_0-1)$ & $\phi_{\alpha,f}(\frac{25}{2}\lambda_1w_{023},\frac{25}{2}\lambda_2w_{013},\frac{25}{2}\lambda_3w_{012})$\\
 & $T$ & $(1,2,1,1)$ & $(5\lambda_1-1)$ & $\phi_{\alpha,f}(\frac{25}{2}\lambda_1w_{023},25\lambda_2w_{013},25\lambda_3w_{012})$\\
 & $T$ & $(1,1,2,1)$ & $(5\lambda_2-1)$ & $\phi_{\alpha,f}(25\lambda_1w_{023},\frac{25}{2}\lambda_2w_{013},25\lambda_3w_{012})$\\
 & $T$ & $(1,1,1,2)$ & $(5\lambda_3-1)$ & $\phi_{\alpha,f}(25\lambda_1w_{023},25\lambda_2w_{013},\frac{25}{2}\lambda_3w_{012})$\\
\bottomrule
\end{tabular}}
\medskip
\caption{Three-dimensional two-form bases ($k=2$) for $r=1,2$. Each block includes triangular-face and tetrahedron-interior functions, with one representative triangular face.}
\label{tab:cardinal-volume}
\end{table}

\Needspace{5\baselineskip}
\begin{remark}[Comparison with the Arnold--Falk--Winther tables]\rm
The bases in~\cite[Tables~9.1--9.2]{ArnoldFalkWinther2009} describe the same trimmed spaces; their orders $2,3$ correspond to $r=1,2$ here. The number of functions attached to each face and the Whitney--bubble factors agree. Their scalar factors are barycentric monomials, whereas ours are shifted Lagrange polynomials normalized for the tangential evaluations in \eqref{eq:nodal-trimmed-dof}.

The change of basis can involve functions attached to containing faces. For example, on $T=[0,1,2]$ with $r=1$, their edge basis on $f=[0,1]$ is $(\lambda_0w_{01},\lambda_1w_{01})$, whereas our function at $\alpha=(2,1)$ is
\[
\Phi_{(2,1),f}^{f,T}=\frac{3}{2}(3\lambda_0-1)w_{01}=3\lambda_0w_{01}-\frac{3}{2}\lambda_1w_{01}-\frac{3}{2}\lambda_2w_{01}.
\]
The last term is a multiple of their triangle-interior basis function. It leaves the edge trace unchanged and makes the shifted factor vanish at the triangle barycenter. Thus the corresponding ambient edge-associated subspaces differ. In contrast, the triangle-interior functions for $r=1$ differ only by the factor $9$. The tables therefore give an explicit basis dual to the chosen point data within the known trimmed spaces; they do not define new element spaces.
\end{remark}

\section{Conclusion}\label{sec:conclusion}

Logarithmic barycentric coordinates connect the Whitney--bubble decomposition
to a metric whose normalized Hodge star realizes the full and trimmed
bubble dualities. The generator formulas give a direct proof of
nondegeneracy of the integral pairings. Shifted Lagrange factors then
yield explicit bases dual to tangential point evaluations in every
dimension and order, including for scalar nodes that are unisolvent on
every face. The resulting interpolant is independent of the roots and
commutes with face restriction. These formulas describe the standard
conforming trimmed spaces; stability and conditioning as the order grows
remain questions for further study.

\section*{Acknowledgments}

The authors thank the organizers and participants of the
\emph{Workshop on Finite Element Tensor Calculus}
(Tsinghua Sanya International Mathematics Forum,
Sanya, January 2026) and the thematic programme
\emph{Differential Complexes: Theory, Discretization,
and Applications}
(Erwin Schr\"odinger International Institute, Vienna, April--June 2026)
for stimulating discussions and valuable opportunities for scientific exchange.

The authors used ChatGPT (OpenAI) for language editing and to assist
in developing the details of the duality arguments and checking the
basis formulas. The authors independently verified all mathematical
statements, proofs, references, and bibliographic information and
take responsibility for the entire manuscript.


\bibliographystyle{abbrv}
\bibliography{FEEC_revised}

\end{document}